\documentclass[leqno,11pt,twoside]{amsart}
\usepackage[body={6.5in, 8.5in},left=1.3in,right=1.3in]{geometry}

\usepackage{times}

\usepackage[all]{xy} % diagrams
\usepackage{amsmath, amssymb, amsfonts, latexsym, mdwlist, amsthm}

\usepackage{url}

\usepackage{amscd}
\usepackage{mathrsfs}
\usepackage{stmaryrd}

\usepackage[colorlinks=true,citecolor=blue, urlcolor=blue, linkcolor=blue]{hyperref}

\newtheorem{theorem}{Theorem}[section]
\newtheorem{lemma}[theorem]{Lemma}
\newtheorem{propo}[theorem]{Proposition}
\newtheorem{defin}[theorem]{Definition}
\newtheorem{corol}[theorem]{Corollary}
\newtheorem{remark}[theorem]{Remark}

\begin{document}

\author{Bingyu Xia}
\address{The Ohio State University, Department of Mathematics, 231 W 18th Avenue, Columbus, OH 43210-1174, USA}
\email{xia.128@osu.edu}



\title{Hilbert scheme of twisted cubics as simple wall-crossing}

\begin{abstract}
We study the Hilbert scheme of twisted cubics in the three-dimensional projective space by using Bridgeland stability conditions. We use wall-crossing techniques to describe its geometric structure and singularities, which reproves the classical result of Piene and Schlessinger.
\end{abstract}

\maketitle
%\tableofcontents

\section{Introduction}
In this paper, we study the birational transformations induced by simple wall-crossings in the space $\mathrm{Stab}(\mathbb{P}^{3})$ of Bridgeland stability conditions on $\mathbb{P}^{3}$ and show how they naturally lead to a new proof of the main result of \cite{PS85,EPS87}.
The notion of stability condition was introduced by Bridgeland in \cite{Bri07}.
It provides a new viewpoint on the study of moduli spaces of sheaves and complexes.
Simple wall-crossings are the most well-behaved wall-crossings in the space of stability conditions.
They are controlled by the extensions of a family of pairs of stable destabilizing objects: they contract a locus of extensions in the moduli of one side of the wall, then produce a new locus of reverse extensions in the moduli of the other side of the wall. The precise definition of a simple wall-crossing is given in Definition \ref{lihai}.
In some examples, the expectation is that a simple wall-crossing will blow up the old moduli space and add a new component that intersects the blow-up transversely along the exceptional locus.
In this paper, we will prove this is indeed the case for the Hilbert scheme of twisted cubics. The main theorem is the following:

\medskip
\noindent\textbf{Main Theorem.} (See also Theorem \ref{ritian}, Theorem \ref{zhu1} and Theorem \ref{zhu2}) \textit{There is a path $\gamma$ in $\mathrm{Stab}(\mathbb{P}^{3})$ that crosses three walls and four chambers for a fixed Chern character $v=\mathrm{ch}(\mathcal{I}_{C})$, where $\mathcal{I_{C}}$ is the ideal sheaf of a twisted cubic $C$. If we list the moduli space of semistable objects in each chamber with respect to the path $\gamma$, we have:}

\textit{$(1)$ The empty space $\emptyset$;}

\textit{$(2)$ A smooth projective integral variety $\mathbf{M}_{1}$ of dimension $12$;}

\textit{$(3)$ A projective variety $\mathbf{M}_{2}$ with two irreducible components $\mathbf{B}$ and $\mathbf{P}$, where $\mathbf{P}$ is a $\mathbb{P}^{9}$-bundle over $\mathbb{P}^{3}\times(\mathbb{P}^{3})^{*}$ and $\mathbf{B}$ is the blow-up of $\mathbf{M}_{1}$ along a $5$-dimensional smooth center. The two components of $\mathbf{M}_{2}$ intersect transversally along the exceptional divisor of $\mathbf{B}$;}

\textit{$(4)$ The Hilbert scheme of twisted cubics $\mathbf{M}_{3}$. $\mathbf{M}_{3}$ is a blow-up of $\mathbf{M}_{2}$ along a $5$-dimensional smooth center contained in $\mathbf{P}\setminus\mathbf{B}$.}
\medskip

Among the above three wall-crossings, the second one and the third one are simple. We are going to study them in great details in Section $4$ and $5$.

In \cite{SchB15}, Schmidt also studied certain wall-crossings on $\mathbb{P}^{3}$. We followed his construction of the path $\gamma$ in the Main Theorem. We will also follow his construction of moduli space $\mathbf{M}_{1}$ by using quiver representations in Section $3$. For the second wall-crossing and the third wall-crossing, Schmidt reinterpreted the main result of \cite{PS85,EPS87} in the new setting of Bridgeland stability. The method of Piene and Schlessinger to study the geometric structure of the Hilbert scheme of twisted cubics is based on deformation theory of ideals. They first used a comparison theorem to show that the Hilbert scheme of twisted cubics is isomorphic to the moduli space of ideals of twisted cubics, and then they use the $\mathbf{P}\mathrm{GL}(4)$-action to reduce tangent space computations to some special ideals. Finally, they exhibited a basis of deformations of these special ideals and computed the versal deformations.

We will use a different method to directly study the second wall-crossing and the third wall-crossing without referring to \cite{PS85,EPS87}. In Section $4$, we first identify the locus $H$ in $\mathbf{M}_{1}$ that is going to be modified after the second wall-crossing. This is Proposition \ref{zhongyao} $(1)$. Then we construct two embeddings of the irreducible components into $\mathbf{M}_{2}$: one is from the projective bundle parametrizing reverse extensions of the family of pairs of destabilizing objects, the other is from the blow-up of $\mathbf{M}_{1}$ along $H$. This is the content of Proposition \ref{zhongyao} $(2)$ and Proposition \ref{ruyao} $(2)$. By definition of a simple wall-crossing, the union of the images of the two embeddings is $\mathbf{M}_{2}$, so $\mathbf{M}_{2}$ only has two irreducible components. With the help of some $\mathrm{Ext}$ computations, we show that the intersection of the two images is the exceptional divisor of the blow-up, and the two embeddings are isomorphisms outside it. This is Remark \ref{fadian}, Remark \ref{mafuyu} (1) and Proposition \ref{ruyao} $(1)$. Finally we study the deformation theory of complexes on the intersection and prove that the two irreducible components of $\mathbf{M}_{2}$ intersect transversely. This is Proposition \ref{gan}. In Section $5$, again we first identify the locus $H'$ that is going to be modified after the third wall-crossing and find that it is solely contained in one irreducible component of $\mathbf{M}_{2}$. Then we construct an isomorphism between the blow-up of $\mathbf{M}_{2}$ along $H'$ and $\mathbf{M}_{3}$, where the latter is the Hilbert scheme of twisted cubics. This is Theorem \ref{haoxiang}. As a consequence, this reproves the main result of \cite{PS85, EPS87} on the geometric structures of the Hilbert scheme of twisted cubics by using stability and wall-crossing techniques. The advantages of this is that we can get rid of using the equations of special ideals. It will be easier sometimes to generalize our approach, especially when the equations are complicated or unavailable.

The Hilbert scheme of twisted cubics is a first nontrivial example where our wall-crossing method applies, and we hope it could be applied in more general cases. Some related works in which our method may apply are: \cite{GHS16} about the moduli of elliptic quartics in $\mathbb{P}^{3}$, \cite{LLMS16} about the moduli of twisted cubics in a cubic fourfold and \cite{Tra16} about the moduli space of certain point-like objects on a surface.

\medskip
\noindent\textbf{Notations.}\begin{align}
\mathrm{Coh}(\mathbb{P}^{3}) &\quad \text{abelian category of coherent sheaves on $\mathbb{P}^{3}$},\nonumber\\
\mathrm{D}^{\mathrm{b}}(\mathbb{P}^{3}) &\quad \text{bounded derived category of $\mathrm{Coh}(\mathbb{P}^{3})$},\nonumber\\
\mathcal{T}_{X} &\quad \text{tangent bundle of a smooth projective variety $X$}\nonumber\\
T_{X,x} &\quad \text{tangent space of $X$ at a point $x$},\nonumber\\
T_{f,x} &\quad \text{tangent map $T_{X,x}\longrightarrow T_{Z,f(x)}$ of a morphism $f:X\longrightarrow Z$},\nonumber\\
\mathcal{N}_{Y/X} &\quad \text{normal bundle of a smooth subvariety $Y$ in $X$},\nonumber\\
N_{Y/X,y} &\quad \text{normal space of $Y$ in $X$ at a point $y$},\nonumber\\
\mathscr{E}xt_{f}^{1}(\mathcal{F},\mathcal{G}) &\quad \text{relative $\mathrm{Ext}^{1}$ sheaf of $\mathcal{F}$ and $\mathcal{G}$ with respect to a morphism $f$},\nonumber\\
\mathscr{T}or^{1}(\mathcal{F},\mathcal{G}) &\quad \text{$\mathrm{Tor}^{1}$ sheaf of $\mathcal{F}$ and $\mathcal{G}$}.\nonumber\\
\mathrm{ch}(E) &\quad \text{Chern charater of an object $E\in\mathrm{D}^{\mathrm{b}}(\mathbb{P}^{3})$}\nonumber\\
c_{i}(E) &\quad \text{$i$-th Chern class of an object $E\in\mathrm{D}^{\mathrm{b}}(\mathbb{P}^{3})$}\nonumber
\end{align}
\noindent\textbf{Acknowledgements.} I would like to thank David Anderson, Yinbang Lin, Giulia Sacc\`{a}, Benjamin Schmidt and Xiaolei Zhao for useful discussions and suggestions. I am in great debt to my advisor Emanuele Macr\`{i}, who introduces me to this topic and gives me advice. I express my  gratitude to the mathematics department of Ohio State University for helping me handle the situation after my advisor moved. I would also like to thank Northeastern University at which most of this paper has been written for their hospitality. The research is partially supported by NSF grants DMS-1302730 and DMS-1523496 (PI Emanuele Macr\`{i}) and a Graduate Special Assignment of the mathematics department of Ohio State University.

\section{A Brief Review on Bridgeland Stability Conditions}
In this section, we review how to construct Bridgeland stability conditions on $\mathbb{P}^{3}$ and define the notion of a simple wall-crossing.
\begin{defin}
A stability condition $(Z,\mathcal{P})$ on $\mathrm{D}^{\mathrm{b}}(\mathbb{P}^{3})$ consists of a group homomorphism $Z:K(\mathrm{D}^{\mathrm{b}}(\mathbb{P}^{3}))\longrightarrow\mathbb{C}$ called central charge, and full additive subcategories $\mathcal{P}(\phi)\subset\mathrm{D}^{\mathrm{b}}(\mathbb{P}^{3})$ for each $\phi\in\mathbb{R}$, satisfying the following axioms:

$(1)$ if $E\in\mathcal{P}(\phi)$ then $Z(E)=m(E)\mathrm{exp}(i\pi\phi)$ for some $m(E)\in\mathbb{R}_{>0}$,

$(2)$ for all $\phi\in\mathbb{R}$, $\mathcal{P}(\phi+1)=\mathcal{P}(\phi)[1]$,

$(3)$ if $\phi_{1}>\phi_{2}$ and $A_{j}\in\mathcal{P}_{j}$, then $\mathrm{Hom}_{\mathrm{D}^{\mathrm{b}}(\mathbb{P}^{3})}(A_{1}, A_{2})=0$,

$(4)$ for each nonzero object $E\in\mathrm{D}^{\mathrm{b}}(\mathbb{P}^{3})$ there are a finite sequence of real numbers
\begin{equation*}
\phi_{1}>\phi_{2}>\cdots>\phi_{n}
\end{equation*}
and a collection of triangles
\begin{displaymath}
\xymatrix {0=E_{0} \ar[r] &E_{1} \ar[d] \ar[r] &E_{2} \ar[d] \ar[r] &\cdots \ar[r] &E_{n-1} \ar[r] &E_{n}=E,\ar[d]\\
& \ar@{.>}[ul] A_{1} & \ar@{.>}[ul] A_{2} & & & \ar@{.>}[ul] A_{n}}
\end{displaymath}
with $A_{j}\in\mathcal{P}(\phi_{j})$ for all $j$.
\end{defin}
If we denote the set of all locally-finite stability conditions by $\mathrm{Stab}(\mathbb{P}^{3})$, then [Theorem 1.2, Bri07] tells us that there is a natural topology on $\mathrm{Stab}(\mathbb{P}^{3})$ making it a complex manifold.

By [Bri07, Proposition $5.3$], to give a stability condition on the bounded derived category of $\mathbb{P}^{3}$, it is equivalent to giving a stability function on a heart of a bounded $t$-structure satisfying the Harder-Narasimhan property. [Tod09, Lemma 2.7] shows this is not possible for the standard heart $\mathrm{Coh}(\mathbb{P}^{3})$. In \cite{BMT14}, stability conditions are constructed on a so-called double tilt $\mathscr{A}^{\alpha,\beta}$ of the standard heart.

We identify the cohomology $\mathrm{H}^{*}(\mathbb{P}^{3},\mathbb{Q})$ with $\mathbb{Q}^{4}$ with respect to the obvious chose of basis. Let $(\alpha,\beta)\in\mathbb{R}_{>0}\times\mathbb{R}$. We define the twisted slope function for $E\in\mathrm{Coh}(\mathbb{P}^{3})$ to be
\begin{equation}
\mu_{\beta}\left(E\right)=\frac{c_{1}\left(E\right)-\beta c_{0}\left(E\right)}{c_{0}\left(E\right)}\nonumber
\end{equation} if $c_{0}(E)\neq0$, and otherwise we let $\mu_{\beta}=+\infty$. Then we set
\begin{align}
\mathcal{T}_{\beta}&=\{E\in\mathrm{Coh}(\mathbb{P}^{3}):\text{any quotient sheaf $G$ of $E$ satisfies }\mu_{\beta}\left(G\right)>0\}\nonumber\\
\mathcal{F}_{\beta}&=\{E\in\mathrm{Coh}(\mathbb{P}^{3}):\text{any subsheaf $F$ of $E$ satisfies }\mu_{\beta}\left(F\right)\leqslant0\}.\nonumber
\end{align}
$(\mathcal{F}_{\beta},\mathcal{T}_{\beta})$ forms a torsion pair in the bounded derived category of $\mathbb{P}^{3}$, because Harder-Narasimhan filtrations exist for the twisted slope $\mu_{\beta}$.
\begin{defin}
Let $\mathrm{Coh}^{\beta}(\mathbb{P}^{3})\subset\mathrm{D}^{\mathrm{b}}(\mathbb{P}^{3})$ be the extension-closure $\langle\mathcal{T}_{\beta}, \mathcal{F}_{\beta}[1]\rangle$. We define the following two functions on $\mathrm{Coh}^{\beta}(\mathbb{P}^{3})$:
\begin{align}
Z_{\alpha,\beta}&=-\left(\mathrm{ch}_{2}-\beta\mathrm{ch}_{1}+\left(\frac{\beta^{2}}{2}-\frac{\alpha^{2}}{2}\right)\mathrm{ch}_{0}\right)+i\left(\mathrm{ch}_{1}-\beta\mathrm{ch}_{0}\right),\nonumber\\
\nu_{\alpha,\beta}&=-\frac{\mathrm{Re}\left(Z_{{\alpha,\beta}}\right)}{\mathrm{Im}\left(Z_{{\alpha,\beta}}\right)}\nonumber
\end{align}
if $\mathrm{Im}(Z_{\alpha,\beta})\neq0$, and we let $\nu_{\alpha,\beta}=+\infty$ otherwise. An object $E\in\mathrm{Coh}^{\beta}(\mathbb{P}^{3})$ is called $\nu_{\alpha,\beta}$-(semi)stable if for all nontrivial subobjects $F$ of $E$, we have $\nu_{\alpha,\beta}(F)<(\leqslant)\nu_{\alpha,\beta}(E/F)$
\end{defin}
An important inequality introduced in \cite{BMT14} and proved in \cite{Mac14} for $\nu_{\alpha,\beta}$-semistable objects is the following:

\begin{theorem}(Generalized Bogomolov-Gieseker inequality) For any $\nu_{\alpha,\beta}$-semistable object $E\in\mathrm{Coh}^{\beta}(\mathbb{P}^{3})$ satisfying $\nu_{\alpha,\beta}(E)=0$, we have the following inequality
\begin{equation}
\mathrm{ch}_{3}\left(E\right)-\beta\mathrm{ch}_{2}\left(E\right)+\frac{\beta^{2}}{2}\mathrm{ch}_{1}\left(E\right)-\frac{\beta^{3}}{6}\mathrm{ch}_{0}\left(E\right)\leqslant\frac{\alpha^{2}}{6}\left(\mathrm{ch}_{1}\left(E\right)-\beta\mathrm{ch}_{0}\left(E\right)\right).\nonumber
\end{equation}
\label{GBG}
\end{theorem}

On the other hand, for the new slope function $\nu_{\alpha,\beta}$, Harder-Narasimhan filtrations also exist. If we repeat the above construction, and define
\begin{align}
\mathcal{T}'_{\alpha,\beta}&=\{E\in\mathrm{Coh}(\mathbb{P}^{3}):\text{any quotient object $G$ of $E$ satisfies }\nu_{\alpha, \beta}(G)>0\}\nonumber\\
\mathcal{F}'_{\alpha,\beta}&=\{E\in\mathrm{Coh}(\mathbb{P}^{3}):\text{any subobject $F$ of $E$ satisfies }\nu_{\alpha, \beta}(F)\leqslant0\}.\nonumber
\end{align}
Then $(\mathcal{F}'_{\alpha,\beta},\mathcal{T}'_{\alpha,\beta})$ forms a torsion pair of $\mathrm{Coh}^{\beta}(\mathbb{P}^{3})$.

\begin{defin}
Let $\mathscr{A}^{\alpha,\beta}\subset\mathrm{D}^{\mathrm{b}}(\mathbb{P}^{3})$ be the extension-closure $\langle\mathcal{T}'_{\alpha,\beta},\mathcal{F}_{\alpha,\beta}[1]\rangle$. We define the following two functions on $\mathscr{A}^{\alpha,\beta}$, for $s>0$:
\begin{align*}
Z_{{\alpha,\beta},s}&=-\left(\mathrm{ch}_{3}-\beta\mathrm{ch}_{2}-\left(\left(s+\frac{1}{6}\right)\alpha^{2}-\frac{\beta^{2}}{2}\right)\mathrm{ch}_{1}-\left(\frac{\beta^{3}}{6}-\left(s+\frac{1}{6}\right)\alpha^{2}\beta\right)\mathrm{ch}_{0}\right)\\&\qquad+i\left(\mathrm{ch}_{2}-\beta\mathrm{ch}_{1}+\left(\frac{\beta^{2}}{2}-\frac{\alpha^{2}}{2}\right)\mathrm{ch}_{0}\right)\\\lambda_{{\alpha,\beta},s}&=-\frac{\mathrm{Re}\left(Z_{{\alpha,\beta},s}\right)}{\mathrm{Im}\left(Z_{{\alpha,\beta},s}\right)}
\end{align*}
if $\mathrm{Im}(Z_{{\alpha,\beta},s})\neq0$, and we let $\lambda_{\alpha,\beta,s}=+\infty$ otherwise. An object $E\in\mathscr{A}^{\alpha,\beta}$ is called $\lambda_{\alpha,\beta,s}$-(semi)stable if for all nontrivial subobjects $F$ of $E$, we have $\lambda_{\alpha,\beta,s}(F)<(\leqslant)\lambda_{\alpha,\beta,s}(E/F)$.
\end{defin}
By [BMT14, Corollary 5.2.4] and [BMS14, Lemma 8.8], Theorem \ref{GBG} implies 
\begin{propo}
The pair ($\mathscr{A}^{\alpha,\beta},Z_{{\alpha,\beta},s}$) is a Bridgeland stability condition on $\mathrm{D}^{\mathrm{b}}(\mathbb{P}^{3})$ for all $(\alpha,\beta,s)\in\mathbb{R}_{>0}\times\mathbb{R}\times\mathbb{R}_{>0}$. The function $(\alpha,\beta,s)\mapsto(\mathscr{A}^{\alpha,\beta},Z_{{\alpha,\beta},s})$ is continuous.
\end{propo}

Once the existence problem is solved, we want to study the moduli space $M_{\lambda_{\alpha,\beta,s}}(v)$ of $\lambda_{\alpha,\beta,s}$-semistable objects $E\in\mathscr{A}^{\alpha,\beta}$ with a fixed Chern character $\mathrm{ch}(E)=v$, and the wall-crossing phenomena in the space of stability conditions when varying $(\alpha,\beta,s)\in\mathbb{R}_{>0}\times\mathbb{R}\times\mathbb{R}_{>0}$. For the wall-crossing phenomena, the expectation here is something similar to [Bri08, Section9]: we have a collection of codimension $1$ submanifolds in $(\alpha,\beta,s)\in\mathbb{R}_{>0}\times\mathbb{R}\times\mathbb{R}_{>0}$ called walls, the complement of all walls is a disjoint union of open subset called chambers. If we move stability conditions in a chamber, there is no strictly semistable object and the set of semistable objects does not change. The set of semistable objects changes only when we cross a wall. For the moduli space of semistable objects, we have two technical difficulties according to \cite{AP06} when we construct it: generic flatness and boundedness. In the case of 3-folds, assuming the generalized Bogomolov-Gieseker inequality, we have the following result from [PT16, Theorem 4.2; Corollary 4.23]:
\begin{theorem}
Assume $X$ is a smooth projective 3-fold on which the generalized Bogomolov-Gieseker inequality holds for tilt-semistable objects, then the moduli functor of Brideland semistable objects $\mathcal{M}_{\sigma}(v)$ for a fixed Chern character $v$ is a quasi-proper algebraic stack of finite-type over $\mathbb{C}$. If there is no strictly semistable object, then $\mathcal{M}_{\sigma}(v)$ is a $\mathbb{C}^{*}$-gerbe over a proper algebraic space $M_{\sigma}(v)$.
\end{theorem}
\noindent There is also an important point in $\mathrm{Stab}(\mathbb{P}^{3})$ called the large volume limit of Bridgeland stability. Roughly speaking, it means when the polarization is large enough (taking $\alpha\rightarrow+\infty$ in Proposition $2.5$), the moduli space of semistable objects will become the same as the moduli space of Gieseker semistable sheaves. [Bri08, Section 14] illustrates this picture in the case of K3 surfaces.

Now we are ready to define the notion of a simple wall-crossing. Fix a wall $W$ and two adjacent chambers $C_{1}$, $C_{2}$ in $\mathrm{Stab}(\mathbb{P}^{3})$, we denote the stability conditions in the chambers $C_{1}$, $C_{2}$ by $\lambda_{1}$, $\lambda_{2}$ respectively.
\begin{defin} A wall-crossing is simple if there exists two nonempty moduli spaces $\mathbf{M}_{A}$ and $\mathbf{M}_{B}$ of semistable objects in $\mathscr{A}^{\alpha,\beta}$ with Chern character $v_{A}$ and $v_{B}$ for stability conditions in a neighborhood of a point on $W$ meeting $C_{1}$ and $C_{2}$ such that:

$(1)$ $v_{A}+v_{B}=v$ and any $A\in\mathbf{M}_{A}$ and $B\in\mathbf{M}_{B}$ is stable;

$(2)$ if $E$ is $\lambda_{1}$-stable but not $\lambda_{2}$-stable, then there exists a unique pair $(A,B)$ in $\mathbf{M}_{A}\times\mathbf{M}_{B}$ such that $0\longrightarrow B\longrightarrow E\longrightarrow A\longrightarrow0$ is a nontrivial extension. Conversely, all nontrivial extensions of $A$ by $B$ are $\lambda_{1}$-stable but not $\lambda_{2}$-stable;

$(3)$ if $F$ is $\lambda_{1}$-stable but not $\lambda_{2}$-stable, then there exists a unique pair $(A,B)$ in $\mathbf{M}_{A}\times\mathbf{M}_{B}$ such that $0\longrightarrow A\longrightarrow F\longrightarrow B\longrightarrow0$ is a nontrivial extension. Conversely, all nontrivial extensions of $B$ by $A$ are $\lambda_{1}$-stable but not $\lambda_{2}$-stable.
\label{lihai}
\end{defin}

Now we fix $v=\mathrm{ch}(\mathcal{I}_{C})$, where $C$ is a twisted cubic in $\mathbb{P}^{3}$. We briefly recall the main ideas of finding the wall-crossings in the Main Theorem without using \cite{PS85,EPS87} as follows: First, we can formally use numerical properties of a wall together with the usual Bogomolov inequality to find the Chern characters $v_{A}$ and $v_{B}$ (Actually, this procedure can be made into a computer algorithm, see [SchB15, Theorem 5.3; Theorem 6.1; Section 5.3] for more details). For the first wall-crossing, we have $v_{A}=\mathrm{ch}(\mathcal{O}(-2)^{3})$ and $v_{B}=\mathrm{ch}(\mathcal{O}(-3)[1]^{2})$. In [SchB15, Proposition 4.5], Schmidt showed that $\mathcal{O}(-2)^{3}$ and $\mathcal{O}(-3)[1]^{2}$ are the only semistable object with those Chern characters. Since these two objects are only strictly semistable, the first wall-crossing is not simple. But it is still not hard to construct the moduli space in this case via quiver representations. For the second wall-crossing, we have $v_{A}=\mathrm{ch}(\mathcal{I}_{p}(-1))$ and $v_{B}=\mathrm{ch}(\mathcal{O}_{V}(-3))$, where $p$ is a point in $\mathbb{P}^{3}$ and $V$ is a plane in $\mathbb{P}^{3}$. In [SchB15, Theorem 5.3], Schmidt showed that $\mathcal{I}_{p}(-1)$ and $\mathcal{O}_{V}(-3)$ are all the semistable objects with those Chern characters. It is also easy to check that in this case $\mathcal{I}_{p}(-1)$ and $\mathcal{O}_{V}(-3)$ are stable, so the second wall-crossing is simple, and the moduli spaces $\mathbf{M}_{A}$ and $\mathbf{M}_{B}$ in Definition \ref{lihai} are $\mathbb{P}^{3}$ and $(\mathbb{P}^{3})^{*}$ respectively. The third wall-crossing is similar to the second wall-crossing. We have $v_{A}=\mathrm{ch}(\mathcal{O}(-1))$ and $v_{B}=\mathrm{ch}(\mathcal{I}_{q/V}(-3))$, where $V$ is a plane in $\mathbb{P}^{3}$ and $q$ is a point on $V$. $\mathcal{O}(-1)$ and $\mathcal{I}_{q/V}(-3)$ are all the semistable objects with those Chern character and they are stable. The third wall-crossing is also simple, with $\mathbf{M}_{A}$ being a point and $\mathbf{M}_{B}$ being the incidence hyperplane $H$ contained in $\mathbb{P}^{3}\times(\mathbb{P}^{3})^{*}$. The statement that $\mathbf{M}_{3}$ is the Hilbert scheme is due to the facts that the large volume limits of Bridgeland stability conditions coincides with Gieseker stability conditions, and the moduli space of Gieseker semistable ideal sheaves is the same with the Hilbert scheme.

We will study the three wall-crossings of the Main Theorem in details in the next three sections.

\section{The First Wall-crossing}
In this section, we construct the moduli space $\mathbf{M}_{1}$ and prove that it is a smooth, projective and integral variety. This part first appears in [SchB15, Theorem 7.1], we will give more details here.

We start with a quiver $Q=(V,A):V=\{v_{1},v_{2}\},A=\{e_{i}|i=1,2,3,4\}$, where $s(e_{i})=v_{1}$ and $t(e_{i})=v_{2}$ (Actually $Q$ is just $\bullet\overset{4}{\longrightarrow}\bullet$). We set a dimension vector to be (2,3) and define $\theta:\mathbb{Z}\oplus\mathbb{Z}\longrightarrow\mathbb{Z}$ to be $\theta(m,n)=-3m+2n$. A representation $V$ with dimension vector $(2,3)$ is $\theta$-(semi)stable if for any proper nontrivial subrepresentation $W$ we have $\theta(\mathrm{\underline{dim}}W)>(\geqslant)0$, where $\mathrm{\underline{dim}}W$ is the dimension vector of $W$. If $S$ is a scheme, we define a family of $\theta$-semistable representations of $Q$ over $S$ with dimension vector $(2,3)$ to be four homomorphisms $f_{0},f_{1},f_{2},f_{3}:V\longrightarrow W$, where $V$ and $W$ are locally free on $S$ with $\mathrm{rk}(V)=2$ and $\mathrm{rk}(W)=3$, such that the representation $f_{0s},f_{1s},f_{2s},f_{3s}:V_{s}\longrightarrow W_{s}$ is $\theta$-semistable for any closed point $s\in S$. We define $\mathcal{K}_{\theta}:\mathbf{Sch}_{\mathbb{C}}\longrightarrow\mathbf{Sets}$ to be the moduli functor sending a scheme $S$ to the set of isomorphism classes of families of $\theta$-semistable representations with dimension vector $(2,3)$ over $S$.
\begin{propo}
The functor $\mathcal{K}_{\theta}$ is represented by a smooth projective integral variety $K_{\theta}$.
\label{hulai}
\end{propo}
\begin{proof}
By \cite{Kin94}, since the dimension vector $(2,3)$ is indivisible, $\mathcal{K}_{\theta}$ is represented by a projective variety $K_{\theta}$ and there is no strictly $\theta$-semistable representation. The path algebra of $Q$ is hereditary since there is no relation between arrows, this means $K_{\theta}$ is smooth and irreducible.
\end{proof}
\begin{theorem}
The two moduli spaces $K_{\theta}$ and $\mathbf{M}_{1}$ are isomorphic.
\label{ritian}
\end{theorem}
\begin{proof}
Fix $(\alpha_{0},\beta_{0})=(\frac{1}{2}+\varepsilon,-\frac{5}{2})$, where $\varepsilon>0$ is small. By [SchB15, Theorem 5.3; Theorem 6.1], $\mathbf{M}_{1}$ is isomorphic to the moduli space $\mathbf{M}^{\mathrm{tilt}}_{\alpha_{0},\beta_{0}}(v)$ of $\nu_{\alpha_{0},\beta_{0}}$-semistable objects in $\mathrm{Coh}^{\beta_{0}}(\mathbb{P}^{3})$. Since $(\alpha_{0},\beta_{0})$ is in the interior of a chamber, there is no strictly semistable objects. Notice that $-3<\beta_{0}<-2$, so by definition $\mathcal{O}(-2)$ and $\mathcal{O}(-3)[1]$ are in $\mathrm{Coh}^{\beta_{0}}(\mathbb{P}^{3})$, and we have
\begin{align}
Z_{\alpha_{0},\beta_{0}}\left(\mathcal{O}(-2)\right)&=-\frac{1}{8}+\frac{\alpha_{0}^{2}}{2}+\frac{1}{2}i,\nonumber\\
Z_{\alpha_{0},\beta_{0}}\left(\mathcal{O}(-3)[1]\right)&=\frac{1}{8}-\frac{\alpha_{0}^{2}}{2}+\frac{1}{2}i.\nonumber
\end{align}
On the other hand, We denote $\mathrm{Rep}(Q)$ to be the abelian category of quiver representations of $Q$, and denote $\mathscr{B}$ to be the extension closure of $\mathcal{O}(-2)$ and $\mathcal{O}(-3)[1]$ in $\mathrm{Coh}^{\beta_{0}}(\mathbb{P}^{3})$. By [SchB15, Theorem 5.1], all $\nu_{\alpha_{0},\beta_{0}}$-semistable objects are in $\mathscr{B}$. By [Bon89, Theorem 6.2], there is an equivalence $F:\mathrm{D}^{\mathrm{b}}(\mathscr{B})\longrightarrow\mathrm{D}^{\mathrm{b}}(\mathrm{Rep}(Q))$. This functor $F$ sends $\mathcal{O}(-3)[1]$ and $\mathcal{O}(-2)$ to the two simple representations $\mathbb{C}\longrightarrow0$ and $0\longrightarrow\mathbb{C}$. On $\mathscr{B}$, we can define a central charge $Z$ and a slope function $\eta$ by
\begin{align*}
Z\left(E\right)&=\theta\left(F^{-1}\left(E\right)\right)+i\mathrm{dim}\left(F^{-1}\left(E\right)\right),\\
\eta\left(E\right)&=-\frac{\mathrm{Re}\left(Z\left(E\right)\right)}{\mathrm{Im}\left(Z\left(E\right)\right)}=-\frac{\theta\left(F^{-1}\left(E\right)\right)}{\mathrm{dim}\left(F^{-1}\left(E\right)\right)},
\end{align*}
where $\mathrm{dim}$ is the sum of the two components of a dimension vector. This will make $\sigma:=(Z,\mathscr{B})$ a stability condition on $\mathrm{D}^{\mathrm{b}}(\mathscr{B})$ by [Bri07, Example 5.5], and $F$ sends $\sigma$-semistable objects with Chern character $v$ to $\theta$-semistable represetations with dimension vector $(2,3)$. If we denote $\mathbf{M}_{\sigma}$ to be the moduli of $\sigma$-semistable objects in $\mathscr{B}$ with Chern character $v$, then actually $F$ defines a bijection map of sets between $\mathbf{M}_{\sigma}$ and $K_{\theta}$. We will globalize this construction later and get a bijective morphism by using the existence of a universal family. Now we compute that
\begin{align}
Z\left(\mathcal{O}(-2)\right)&=2+i,\nonumber\\
Z\left(\mathcal{O}(-3)[1]\right)&=-3+i.\nonumber
\end{align}
If we view $Z$ and $Z_{\alpha_{0},\beta_{0}}|_{\mathrm{D}^{\mathrm{b}}(\mathscr{B})}$ as linear maps from $\mathbb{Z}^{2}$ to $\mathbb{R}^{2}$, then an easy computation shows they differ from each other by composing a linear map in $\mathrm{GL}^{+}(2;\mathbb{R})$. This means they define the same stability condition and hence have the same moduli of semistable objects with Chern character $v$, so $\mathbf{M}_{\sigma}=\mathbf{M}^{\mathrm{tilt}}_{\alpha_{0},\beta_{0}}(v)$.

It only remains to show that $K_{\theta}$ is isomorphic to $\mathbf{M}_{\sigma}$. For any $\sigma$-semistable object $E\in\mathrm{D}^{\mathrm{b}}(\mathscr{B})$ with Chern character $v$, $F(E)$ is a $\theta$-semistable representation $f_{1}, f_{2}, f_{3}, f_{4}:\mathbb{C}^{3}\longrightarrow\mathbb{C}^{2}$. We have an obvious exact sequence
\begin{center}
$\begin{CD}
0 @>>> \mathbb{C}^{3} @>>> \mathbb{C}^{3}\\
@VVV @V f_{i}VV @VVV \\
\mathbb{C}^{2} @>>> \mathbb{C}^{2} @>>> 0
\end{CD}$
\end{center}
in $\mathrm{Rep}(Q)$ which corresponds to an exact sequence $\mathcal{O}(-2)^{3}\longrightarrow E\longrightarrow\mathcal{O}(-3)[1]^{2}$ in $\mathscr{B}$. By applying the long exact sequence for $\mathrm{Hom}$ functor to it, we can see that $\mathrm{Ext}^{2}(E,E)=0$. But $\mathrm{Ext}^{2}(E,E)$ computes the obstruction space of $\mathbf{M}_{\sigma}$ at $E$ by \cite{Ina02} and \cite{Lie06}, so $\mathbf{M}_{\sigma}$ is smooth and hence a complex manifold. Since there is no strictly $\sigma$-semistable object, a universal family $\mathcal{U}$ of $\sigma$-semistable objects with Chern character $v$ exists on $\mathbf{M}_{\sigma}\times\mathbb{P}^{3}$, and $\mathcal{U}$ is an extension of $p^{*}\mathcal{O}(-3)^{\oplus2}[1]$ by $p^{*}\mathcal{O}(-2)^{\oplus3}$. If we denote $\mathscr{B}'$ to be the extension closure of $p^{*}\mathcal{O}(-3)^{\oplus2}[1]$ and $p^{*}\mathcal{O}(-2)^{\oplus3}$ in $\mathrm{D}^{\mathrm{b}}(\mathbf{M}_{\sigma}\times\mathbb{P}^{3})$, and denote $\mathrm{Rep}_{K_{\theta}}(Q)$ to be the category of families of quiver representations over $K_{\theta}$. Then there exists an equivalence $F_{K_{\theta}}:\mathscr{B}'\longrightarrow\mathrm{D}^{\mathrm{b}}(\mathrm{Rep}_{K_{\theta}}(Q))$ such that when restricted to a fiber $x\times\mathbb{P}^{3}$, $F_{K_{\theta}}$ is the same with $F$. Because $F_{K_{\theta}}(\mathcal{U})|_{x\times\mathbb{P}^{3}}=F(\mathcal{U}|_{x\times\mathbb{P}^{3}})$ and $\mathcal{U}|_{x\times\mathbb{P}^{3}}$ is a $\sigma$-semistable object with Chern character $v$, $F_{K_{\theta}}(\mathcal{U})|_{x\times\mathbb{P}^{3}}$ is $\theta$-semistable with dimension vector $(2,3)$. This means $F_{K_{\theta}}(\mathcal{U})$ is a family of $\theta$-semistable objects with dimension vector $(2,3)$, so it induces a morphism $\varphi:\mathbf{M}_{\sigma}\longrightarrow K_{\theta}$. As $\mathcal{U}$ is a universal family of $\sigma$-semistable objects with Chern character $v$, and $F$ is a bijection between $\sigma$-semistable objects with Chern character $v$ in $\mathscr{B}$ and $\theta$-semistable representations with dimension vector $(2,3)$, $\varphi$ is a bijective morphism. We proved that $K_{\theta}$ is smooth in Proposition $3.1$, and any bijiective morphism between complex manifolds is an isomorphism, so $\varphi$ is an isomorphism. Therefore $K_{\theta}$ is isomorphic to $\mathbf{M}_{1}$.
\end{proof}

\section{The Second Wall-crossing}
In this section, we study the second wall-crossing and prove $(3)$ in the Main Theorem. To be more precise, we will prove the following theorem. Let $V$ be a plane in $\mathbb{P}^{3}$ and $p$ be a point in $\mathbb{P}^{3}$.
\begin{theorem}
The second wall-crossing is simple with a family of pairs of destabilizing objects $(\mathcal{I}_{p}(-1)$, $\mathcal{O}_{V}(-3))$. The moduli space of semistable objects after the wall-crossing is a projective variety $\mathbf{M}_{2}$. $\mathbf{M}_{2}$ has two irreducible components $\mathbf{B}$ and $\mathbf{P}$, where $\mathbf{P}$ is a $\mathbb{P}^{9}$-bundle over $\mathbb{P}^{3}\times(\mathbb{P}^{3})^{*}$ and $\mathbf{B}$ is the blow-up of $\mathbf{M}_{1}$ along a $5$-dimensional smooth center. The two components of $\mathbf{M}_{2}$ intersect transversally along the exceptional divisor of $\mathbf{B}$.
\label{zhu1}
\end{theorem}
Throughout this section, we fix the family of pairs of destabilizing objects to be 
\begin{equation}
\left(A, B\right)=\left(\mathcal{I}_{p}(-1),\mathcal{O}_{V}(-3)\right),\nonumber
\end{equation}and denote the stability conditions in the chamber of $\mathbf{M}_{1}$ (resp. $\mathbf{M}_{2}$) by $\lambda_{1}$ (resp. $\lambda_{2}$). Whenever we take an extension of $A$ and $B$, we always mean a nontrivial extension class modulo scalar multiplications. The following Hom and Ext group computations are straightforward.
\begin{lemma}
$\mathrm{Hom}(A,B)=\mathrm{Hom}(B,A)=0$, $\mathrm{Hom}(A,A)=\mathrm{Hom}(B,B)=\mathbb{C}$;

$\mathrm{Ext}^{1}(A,B)=\mathbb{C}$ if $p\in V$, and $0$ otherwise,
 
$\mathrm{Ext}^{1}(A,A)=\mathrm{Ext}^{1}(B,B)=\mathbb{C}^{3}$, $\mathrm{Ext}^{1}(B,A)=\mathbb{C}^{10}$; 
 
$\mathrm{Ext}^{2}(A,B)=\mathbb{C}$, $\mathrm{Ext}^{2}(B,B)=0$, $\mathrm{Ext}^{2}(A,A)=\mathbb{C}^{3}$, $\mathrm{Ext}^{2}(B,A)=0$;
 
$\mathrm{Ext}^{3}(A,B)=\mathrm{Ext}^{3}(A,A)=\mathrm{Ext}^{3}(B,B)=\mathrm{Ext}^{3}(B,A)=0$.
\label{diao}
\end{lemma}
\smallskip

\noindent\textbf{Moduli space of nontrivial extensions.} In this subsection, we construct two moduli spaces $H$ and $\mathbf{P}$, where $H$ parametrizes nontrivial extensions of $A$ by $B$ and $\mathbf{P}$ parametrizing the reverse nontrivial extensions. We show that with the universal extensions on those moduli spaces, $H$ is embedded into $\mathbf{M}_{1}$ and $\mathbf{P}$ is embedded into $\mathbf{M}_{2}$. Then we do some detailed comutations on $\mathrm{Ext}$ groups for later uses.

We recall the comments after Definition \ref{lihai}: the second wall-crossing is simple and we have $\mathbf{M}_{A}=\mathbb{P}^{3}$ parametrizing $\mathcal{I}_{p}(-1)$ and $\mathbf{M}_{B}=(\mathbb{P}^{3})^{*}$ parametrizing $\mathcal{O}_{V}(-3)$. We denote the universal family of semistable objects with Chern character $v_{A}$ on $\mathbf{M}_{A}\times\mathbb{P}^{3}$ by $\mathcal{U}_{A}$, and the universal family of semistable objects with Chern character $v_{B}$ on $\mathbf{M}_{B}\times\mathbb{P}^{3}$ by $\mathcal{U}_{B}$. Denote two projections by
\begin{equation*}
\mathbf{M}_{A}\times\mathbb{P}^{3}\overset{\pi_{A}}{\longleftarrow}\mathbf{M}_{A}\times\mathbf{M}_{B}\times\mathbb{P}^{3}\overset{\pi_{B}}{\longrightarrow}\mathbf{M}_{B}\times\mathbb{P}^{3}.
\end{equation*}
We also denote the projection onto the first two factors by $\mathbf{M}_{A}\times\mathbf{M}_{B}\times\mathbb{P}^{3}\overset{\pi}{\longrightarrow}\mathbf{M}_{A}\times\mathbf{M}_{B}$. Let $H$ be the incidence hyperplane $\{(p,V)\in\mathbb{P}^{3}\times(\mathbb{P}^{3})^{*}|p\in V\}$, and denote the restriction of the above three projections to $H\times\mathbb{P}^{3}$ by $\pi_{A}^{H}$, $\pi_{B}^{H}$ and $\pi_{H}$. Define $\mathcal{F}$ to be $\pi_{A}^{*}\mathcal{U}_{A}$ and $\mathcal{G}$ to be $\pi_{B}^{*}\mathcal{U}_{B}$, and define $\mathcal{F}_{H}$ to be $\left(\pi_{A}^{H}\right)^{*}\mathcal{U}_{A}$ and $\mathcal{G}_{H}$ to be $\left(\pi_{B}^{H}\right)^{*}\mathcal{U}_{B}$. Let $S\longrightarrow\mathbf{M}_{A}\times\mathbf{M}_{B}$ and $S_{H}\longrightarrow H$ be any morphisms of schemes, and denote the pullbacks of these two morphisms with respect to $\pi$ and $\pi_{H}$ by $q^{S}$ and $q_{H}^{S}$.

\begin{propo}
There exists an extension on $H\times\mathbb{P}^{3}$
\begin{equation}
0\longrightarrow\mathcal{G}_{H}\otimes \pi_{H}^{*}\mathcal{L}\longrightarrow\mathcal{U}_{E}\longrightarrow\mathcal{F}_{H}\longrightarrow0,
\end{equation}
$\mathcal{L}=\mathscr{E}xt^{1}_{\pi_{H}}(\mathcal{F}_{H},\mathcal{G}_{H})^{*}$ is a line bundle, which is universal on the category of noetherian $H$-schemes for the classes of nontrivial extensions of $\left(q_{H}^{S}\right)^{*}\mathcal{F}_{H}$ by $\left(q_{H}^{S}\right)^{*}\mathcal{G}_{H}$ on $\left(H\times\mathbb{P}^{3}\right)\times_{H}S_{H}$, modulo the scalar mutiplication of $H^{0}(S_{H},\mathcal{O}_{S_{H}}^{*})$.
\label{dadiao}
\end{propo}
\begin{proof}
We apply [Lan85, Proposition 4.2; Corollary 4.5] to $\mathcal{F}_{H}$, $\mathcal{G}_{H}$ and $\pi_{H}$. We only need to check that $\mathscr{E}xt^{0}_{\pi_{H}}(\mathcal{F}_{H},\mathcal{G}_{H})=0$ and $\mathscr{E}xt^{1}_{\pi_{H}}(\mathcal{F}_{H},\mathcal{G}_{H})$ commutes with base change in the sense that over any point $(p_{0},V_{0})\in H$, $\mathscr{E}xt^{1}_{\pi_{H}}(\mathcal{F}_{H},\mathcal{G}_{H})$ restricts to $\mathrm{Ext}^{1}(A_{0},B_{0})$. First notice that $\mathscr{E}xt^{3}_{\pi_{H}}(\mathcal{F}_{H},\mathcal{G}_{H})$ restricts to $\mathrm{Ext}^{3}(A_{0}, B_{0})$ over $(p_{0},V_{0})$, where the latter is $0$ by Lemma $4.1$. Then [Lan85, Theorem 1.4] tells us $\mathscr{E}xt^{2}_{\pi_{H}}(\mathcal{F}_{H},\mathcal{G}_{H})$ restricts to $\mathrm{Ext}^{2}(A_{0}, B_{0})$ over $(p_{0},V_{0})$, where the latter is $\mathbb{C}$ for all points in $H$. Hence $\mathscr{E}xt^{2}_{\pi_{H}}(\mathcal{F}_{H},\mathcal{G}_{H})$ is a line bundle. Again [Lan85, Theorem 1.4] tells us $\mathscr{E}xt^{1}_{\pi_{H}}(\mathcal{F}_{H},\mathcal{G}_{H})$ restricts to $\mathrm{Ext}^{1}(A_{0}, B_{0})$ over $(p_{0},V_{0})$. By Lemma $4.1$ we have $\mathrm{Ext}^{1}(A_{0}, B_{0})=\mathbb{C}$ for all points in $H$, so $\mathscr{E}xt^{1}_{\pi_{H}}(\mathcal{F}_{H},\mathcal{G}_{H})$ is a line bundle. Applying [Lan85, Theorem 1.4] a third time, $\mathscr{E}xt^{0}_{\pi_{H}}(\mathcal{F}_{H},\mathcal{G}_{H})$ will restrict to $\mathrm{Hom}(A_{0},B_{0})$, where the latter is $0$ by Lemma $4.1$. Hence $\mathscr{E}xt^{0}_{\pi_{H}}(\mathcal{F}_{H},\mathcal{G}_{H})=0$.
\end{proof}
\begin{propo}
The relative Ext sheaf $\mathscr{E}xt^{1}_{\pi}(\mathcal{G},\mathcal{F})$ is locally free of rank $10$ on $\mathbf{M}_{A}\times\mathbf{M}_{B}$. If we denote its projectivization $\mathbb{P}(\mathscr{E}xt^{1}_{\pi}(\mathcal{G},\mathcal{F})^{*})$ by $\mathbf{P}$, then there exists an extension on $\mathbf{P}\times\mathbb{P}^{3}$
\begin{equation}
0\longrightarrow h^{*}\mathcal{F}\otimes\pi_{\mathbf{P}}^{*}\mathcal{O}_{\mathbf{P}}(1)\longrightarrow\mathcal{U}_{F}\longrightarrow h^{*}\mathcal{G}\longrightarrow0,
\end{equation}
$h$ is the projection $\mathbf{P}\times\mathbb{P}^{3}\longrightarrow\mathbf{M}_{A}\times\mathbf{M}_{B}\times\mathbb{P}^{3}$, $\pi_{\mathbf{P}}$ is the projection $\mathbf{P}\times\mathbb{P}^{3}\longrightarrow \mathbf{P}$ and $\mathcal{O}_{\mathbf{P}}(1)$ is the relative $\mathcal{O}(1)$ on $\mathbf{P}$, which is universal on the category of noetherian $\mathbf{M}_{A}\times\mathbf{M}_{B}$-schemes for the classes of nontrivial extensions of $\left(q^{S}\right)^{*}\mathcal{F}$ by $\left(q^{S}\right)^{*}\mathcal{G}$ on $\left(\mathbf{M}_{A}\times\mathbf{M}_{B}\times\mathbb{P}^{3}\right)\times_{\mathbf{M}_{A}\times\mathbf{M}_{B}}S$, modulo the scalar mutiplication of $H^{0}(S,\mathcal{O}_{S}^{*})$.
\label{juru}
\end{propo}
\begin{proof}
The proof is completely analogous to the proof of Proposition \ref{dadiao}.
\end{proof}

The existence the above extension $\mathcal{U}_{E}$ (resp. $\mathcal{U_{F}}$) gives a flat family of $\lambda_{1}$-stable (resp. $\lambda_{2}$-stable) sheaves on $H$ (resp. $\mathbf{P}$), hence it induces a morphism $\varphi_{E}:H\longrightarrow \mathbf{M}_{1}$ (resp. $\varphi_{F}:\mathbf{P}\longrightarrow \mathbf{M}_{2}$).

\begin{propo}
(1) The induced morphism $\varphi_{E}$ is a closed embedding;

(2) The induced morphism $\varphi_{F}$ is injective on the level of sets and Zariski tangent spaces.
\label{zhongyao}
\end{propo}
\begin{proof}
On the level of sets, $\varphi_{E}$ maps an extension $0\longrightarrow B\longrightarrow E\longrightarrow A\longrightarrow0$ to $E$. If we have two extensions $0\longrightarrow B\longrightarrow E\longrightarrow A\longrightarrow0$ and $0\longrightarrow B'\longrightarrow E'\longrightarrow A'\longrightarrow0$ such that $E\cong E'$ as stable sheaves, then $E'=E$ and this isomorphism is just a  scalar multiplication by some $c\in\mathbb{C}^{*}$. By the definition of a simple wall-crossing with a pair of destabilizing object, we must have $A'=A$ and $B'=B$. This implies that $\varphi_{E}$ is injective on the level of sets.

On the level of Zariski tangent spaces, a tangent vector $v$ of $H$ at a point $(p,V)$ can be represented by a morphism $\mathrm{Spec}\mathbb{C}[\varepsilon]/(\varepsilon^{2})\longrightarrow H$. By pulling back the universal extension $(1)$ to $\left(H\times\mathbb{P}^{3}\right)\times_{H}\mathrm{Spec}\mathbb{C}[\varepsilon]/(\varepsilon^{2})=\mathrm{Spec}\mathbb{C}[\varepsilon]/(\varepsilon^{2})\times\mathbb{P}^{3}$, we get an exact sequence of flat families
\begin{equation*}
0\longrightarrow\mathcal{G}_{\varepsilon}\longrightarrow\mathcal{E}_{\varepsilon}\longrightarrow\mathcal{F}_{\varepsilon}\longrightarrow0
\end{equation*}
and $\mathcal{G}_{\varepsilon}$, $\mathcal{E}_{\varepsilon}$ and $\mathcal{F}_{\varepsilon}$ restrict to $B$, $E$ and $A$ on the closed fiber respectively. In particular, $\mathcal{E}_{\varepsilon}$ is a flat family of $\lambda_{1}$-stable objects. It gives rise to a morphism $\mathrm{Spec}\mathbb{C}[\varepsilon]/(\varepsilon^{2})\longrightarrow\mathbf{M}_{1}$ corresponding to $T_{\varphi_{E},(p,V)}(v)$. Suppose we have two tangent vectors $v$, $v'$ represented by morphisms $\xi,\xi':\mathrm{Spec}\mathbb{C}[\varepsilon]/(\varepsilon^{2})\longrightarrow H$ and $T_{\varphi_{E},(p,V)}(v)=T_{\varphi_{E},(p,V)}(v')$. Then there exists an isomorphism $\eta:\mathcal{E}_{\varepsilon}\longrightarrow\mathcal{E}_{\varepsilon}'$ between the resulting flat families of $\lambda_{1}$-stable objects such that $\eta$ restricts to identity on the closed fiber. By \cite{Ina02} and \cite{Lie06}, $\eta$ corresponds to the following diagram in the derived category:
\begin{center}
$\begin{CD}
E @= E \\
@V\zeta VV @V\zeta'VV \\
E[1] @>c>> E[1],
\end{CD}$
\end{center}
where $c$ is a multiplication by some nonzero constant $c$. By composing $\xi$ and $\xi'$ with the natural projections \begin{equation*}
\mathbf{M}_{A}=\mathbb{P}^{3}\longleftarrow H\longrightarrow(\mathbb{P}^3)^{*}=\mathbf{M}_{B},
\end{equation*}
we can complete $\zeta$ and $\zeta'$ to commutative diagrams
\begin{center}
$\begin{CD}
B @>>> E @>>> A @.\qquad B @>>> E @>>> A \\
@VVV @V\zeta VV @VVV \qquad @VVV @V\zeta' VV @VVV \\
B[1] @>>> E[1] @>>> A[1] @.\qquad B[1] @>>> E[1] @>>> A[1],
\end{CD}$
\end{center}
Via the two diagrams, the above diagram of $\eta$ will induce two diagrams
\begin{center}
$\begin{CD}
B @= B @.\qquad A @= A \\
@V\zeta_{B}VV @V\zeta'_{B}VV \qquad @V\zeta_{A} VV @V\zeta'_{A}VV \\
B[1] @>c>> B[1] @.\qquad A[1] @>c>> A[1]
\end{CD}$
\end{center}
corresponding to isomorphisms $\eta_{B}:\mathcal{G}_{\varepsilon}\longrightarrow\mathcal{G}_{\varepsilon}'$ and $\eta_{A}:\mathcal{F}_{\varepsilon}\longrightarrow\mathcal{F}_{\varepsilon}'$ such that they restrict to identities on closed fiber and they make the following diagram commutative:
\begin{center}
$\begin{CD}
0@>>>\mathcal{G}_{\varepsilon}@>>>\mathcal{E}_{\varepsilon}@>>>\mathcal{F}_{\varepsilon}@>>>0\\
@. @V\eta_{B}VV @V\eta VV @V\eta_{A}VV @.\\
0@>>>\mathcal{G}_{\varepsilon}'@>>>\mathcal{E}_{\varepsilon}'@>>>\mathcal{F}_{\varepsilon}'@>>>0,
\end{CD}$
\end{center}
which implies the two morphisms $\xi$ and $\xi'$ are the same. Therefore $v=v'$ and $T_{\varphi_{E},E}$ is injective. This proves that $\varphi_{E}$ is a closed embedding. The proof of (2) is completely analogous to the above argument.
\end{proof}
Now we study the normal sequence of the embedding $\varphi_{E}:H\longrightarrow \mathbf{M}_{1}$. Fix a nontrivial extension $0\longrightarrow B\longrightarrow E\longrightarrow A\longrightarrow0$, then we have the following lemma.
\begin{lemma}
The following diagram is coming from taking the long exact sequences for $\mathrm{Hom}$ functor in two directions, it is commutative with exact rows and columns and all boundary homomorphisms are $0$.
\label{zheteng}
\end{lemma}
\noindent$\begin{CD}
\mathrm{Ext}^{1}(A,B)=\mathbb{C} @>0>> \mathrm{Ext}^{1}(A,E)=\mathbb{C}^{2} @>>> \mathrm{Ext}^{1}(A,A)=\mathbb{C}^{3} @>>> \mathrm{Ext}^{2}(A,B)=\mathbb{C} \\
@V0VV @VVV @VVV @VVV \\
\mathrm{Ext}^{1}(E,B)=\mathbb{C}^{2} @>>> \mathrm{Ext}^{1}(E,E)=\mathbb{C}^{12} @>>> \mathrm{Ext}^{1}(E,A)=\mathbb{C}^{10} @>>> \mathrm{Ext}^{2}(E,B)=0 \\
@VVV @VVV @VVV @VVV \\
\mathrm{Ext}^{1}(B,B)=\mathbb{C}^{3} @>>> \mathrm{Ext}^{1}(B,E)=\mathbb{C}^{13} @>>> \mathrm{Ext}^{1}(B,A)=\mathbb{C}^{10} @>>> \mathrm{Ext}^{2}(B,B)=0 \\
@VVV @VVV @VVV @VVV \\
\mathrm{Ext}^{2}(A,B)=\mathbb{C} @>0>> \mathrm{Ext}^{2}(A,E)=\mathbb{C}^{3} @>>> \mathrm{Ext}^{2}(A,A)=\mathbb{C}^{3} @>>> 0
\end{CD}$
\medskip

\begin{proof}
This diagram is a straightforward computation by using that $(A, B)=(\mathcal{I}_{p}(-1)$, $\mathcal{O}_{V}(-3))$ and that $E$ satisfies a triangle $\mathcal{O}(-2)^{3}\longrightarrow E\longrightarrow\mathcal{O}(-3)[1]^{2}$.
\end{proof}
The Kodaira-Spencer map $\mathrm{KS}:T_{\mathbf{M}_{1},E}\longrightarrow\mathrm{Ext}^{1}(E,E)$ is known to be an isomorphism by \cite{Ina02} and \cite{Lie06}. If we let $\theta_{E}$ to be the composition $
\mathrm{Ext}^{1}(E,E)\longrightarrow\mathrm{Ext}^{1}(E,A)\longrightarrow\mathrm{Ext}^{1}(B,A)$ (or $\mathrm{Ext}^{1}(E,E)\longrightarrow\mathrm{Ext}^{1}(B,E)\longrightarrow\mathrm{Ext}^{1}(B,A)$) in the diagram of Lemma \ref{zheteng}, and let the kernel of $\theta_{E}$ to be $K_{E}$, then we have
\begin{propo}
The Kodaira-Spencer map $\mathrm{KS}$ restricts to an isomorphism between $T_{H,E}$ and $K_{E}$, and we have the following commutative diagram:
\label{KS}
\end{propo}
\begin{center}
$\begin{CD}
0 @>>> T_{H,E} @>>> T_{\mathbf{M}_{1},E} @>>> N_{H/\mathbf{M}_{1},E} @>>> 0\\
  @.   @VV\mathrm{KS}V         @VV\mathrm{KS}V @VVV\\
0 @>>> K_{E} @>>> \mathrm{Ext}^{1}(E,E) @>\theta_{E}>> \mathrm{Ext}^{1}(B,A)
\end{CD}$
\end{center}
\begin{proof}  
$\theta_{E}$ is the composition of $\mathrm{Ext}^{1}(E,E)\longrightarrow\mathrm{Ext}^{1}(E,A)\longrightarrow\mathrm{Ext}^{1}(B,A)$, where the first map is surjective with a two-dimensional kernel $\mathrm{Ext}^{1}(E,B)$ and the second map has a $3$-dimensional kernel $\mathrm{Ext}^{1}(A,A)$ by Lemma \ref{zheteng}. This implies $K_{E}$ is $5$-dimensional since $K_{E}$ is an extension of $\mathrm{Ext}^{1}(A,A)$ by $\mathrm{Ext}^{1}(E,B)$, so $\mathrm{dim}K_{E}=\mathrm{dim}T_{H,E}$. On the other hand, as shown in the proof of Proposition \ref{zhongyao}, a vector $v$ in $T_{H,E}$ is represented by a commutative diagram:
\begin{center}
$\begin{CD}
 B @>>> E @>>> A \\
 @VVV @V\mathrm{KS}(v) VV @VVV  \\
B[1] @>>> E[1] @>>> A[1]
\end{CD}$.
\end{center}
$\theta_{E}(\mathrm{KS}(v))$ is equal to the composition $B\longrightarrow E\overset{\mathrm{KS}(v)}{\longrightarrow} E[1]\longrightarrow A[1]$, which is zero since by using the commutativity of the diagram. Hence $T_{H,E}$ is mapped into $K_{E}$ under $\mathrm{KS}$. Since we have proved $\mathrm{dim}K_{E}=\mathrm{dim}T_{H,E}$ , $\mathrm{KS}$ canonically induces an isomorphism between them.
\end{proof}

We can also define $\theta_{F}:\mathrm{Ext}^{1}(F,F)\longrightarrow\mathrm{Ext}^{1}(A,B)$ for any nontrivial extension $0\longrightarrow A\longrightarrow F\longrightarrow B\longrightarrow0$ in a similar way. Denote its kernel by $K_{F}$, then we have :
\begin{corol}
The tangent space $T_{\mathbf{P},F}$ is canonically identified with $K_{F}$ under the Kodaira-Spencer map.
\label{ciyao}
\end{corol}
\begin{proof}
The reason that $T_{\mathbf{P},F}$ is mapped into $K_{F}$ under the Kodaira-Spencer map is the same as in the case of Proposition \ref{KS}. Conversely, take any $\zeta\in K_{F}$, we have that the composition $A\longrightarrow F\overset{\zeta}{\longrightarrow} F[1]\longrightarrow B[1]$ is $0$. By using the universal property of a triangle in the derived category, there exists morphisms $A\longrightarrow A[1]$ and $B\longrightarrow B[1]$ such that the following diagram is commutative:
\begin{center}
$\begin{CD}
 A @>>> F @>>> B \\
 @VVV @V\zeta VV @VVV  \\
A[1] @>>> F[1] @>>> B[1]
\end{CD}$.
\end{center}
This diagram will correspond to an exact sequence of flat families on $\mathrm{Spec}\mathbb{C}[\varepsilon]/(\varepsilon^{2})\times\mathbb{P}^{3}$
\begin{equation*}
0\longrightarrow\mathcal{F}_{\varepsilon}\longrightarrow\mathcal{F}_{\varepsilon}'\longrightarrow\mathcal{G}_{\varepsilon}\longrightarrow0
\end{equation*}
where $\mathcal{F}_{\varepsilon}$, $\mathcal{F}_{\varepsilon}'$ and $\mathcal{G}_{\varepsilon}$ will restrict to $A$, $F$ and $B$ on the closed fiber. By the universal property of $\mathbf{P}$ proved in Proposition \ref{juru}, this sequence induces a morphism from $\mathrm{Spec}\mathbb{C}[\varepsilon]/(\varepsilon^{2})$ to $\mathbf{P}$ corresponding to a tangent vector $v$ of $\mathbf{P}$ at $F$. It is not hard to check $\mathrm{KS}(v)=\zeta$, so $\mathrm{KS}$ is also surjective between $T_{\mathbf{P},F}$ and $K_{F}$.
\end{proof}

We can use the exact sequence (1) to write down the following globalization of the diagram in Proposition \ref{KS}.
\begin{propo}
The following diagram has exact rows. Among the three vertical morphisms, the left one and middle one are isomorphisms, and the right one is an injection.
\begin{displaymath}
\xymatrix{0 \ar[r] & \mathcal{T}_{H} \ar[r] \ar[d] &\mathcal{T}_{\mathbf{M}_{1}}|_{H} \ar[r] \ar[d]_{\mathrm{KS}} &\mathcal{N}_{H/\mathbf{M}_{1}} \ar[r] \ar[d] &0\\
0 \ar[r] & \mathcal{K}_{E} \ar[r] & \mathscr{E}xt^{1}_{\pi_{H}}(\mathcal{U}_{E},\mathcal{U}_{E}) \ar[r] &\mathscr{E}xt^{1}_{\pi_{H}}(\mathcal{G}_{H}\otimes \pi_{H}^{*}\mathcal{L},\mathcal{F}_{H})}
\end{displaymath}
\label{duodiao}
\end{propo}
From this proposition we see that the normal bundle $\mathcal{N}_{H/\mathbf{M}_{1}}$ embedds into $\mathscr{E}xt^{1}_{\pi_{H}}(\mathcal{G}_{H}\otimes\pi_{H}^{*}\mathcal{L},\mathcal{F}_{H})$, hence its projectivization $\mathbb{P}(\mathcal{N}_{H/\mathbf{M}_{1}}^{*})$ is embedded in $\mathbb{P}(\mathscr{E}xt^{1}_{\pi_{H}}(\mathcal{G}_{H}\otimes\pi_{H}^{*}\mathcal{L},\mathcal{F}_{H})^{*})=\mathbb{P}(\mathscr{E}xt^{1}_{\pi_{H}}(\mathcal{G}_{H},\mathcal{F}_{H})^{*})$, where the latter is the preimage of $H$ under the projection $\mathbb{P}(\mathscr{E}xt^{1}_{\pi}(\mathcal{G},\mathcal{F})^{*})=\mathbf{P}\longrightarrow\mathbb{P}^3\times(\mathbb{P}^{3})^{*}$.

Next we are going to compute the dimension of the Zariski tangent space $T_{\mathbf{M}_{2},F}\cong\mathrm{Ext}^{1}(F,F)$ for a nontrivial extension $0\longrightarrow A\longrightarrow F\longrightarrow B\longrightarrow0$. First let us introduce some notations: we denote $e:A\longrightarrow B[1]$ the nontrivial extension of $A$ by $B$ and name the arrows $ B\overset{h}{\longrightarrow}E\overset{j}{\longrightarrow}A$. Similarly let $f:B\longrightarrow A[1]$ be the extension we fix and name the arrows $ A\overset{k}{\longrightarrow}F\overset{l}{\longrightarrow}B$. There are three cases and they are taken care of by the following three propositions.

\begin{propo}
If $F\in\mathbb{P}(\mathcal{N}_{H/\mathbf{M}_{1}}^{*})$, then we have the following commutative diagram with exact rows and columns. All boundary homomorphisms are $0$ except at $\mathrm{Ext}^{1}(B,A)$, where the two homomophisms $\mathrm{Ext}^{1}(F,A)\longleftarrow\mathrm{Ext}^{1}(B,A)\longrightarrow\mathrm{Ext}^{1}(B,F)$ have a same $1$-dimensional kernel $\mathbb{C}f$.
\label{cao}
\end{propo}
\begin{center}
$\begin{CD}
\mathrm{Ext}^{1}(B,A)=\mathbb{C}^{10} @>>> \mathrm{Ext}^{1}(F,A)=\mathbb{C}^{12} @>>> \mathrm{Ext}^{1}(A,A)=\mathbb{C}^{3}\\
@VVV @VVV @VVV\\
\mathrm{Ext}^{1}(B,F)=\mathbb{C}^{12} @>>> \mathrm{Ext}^{1}(F,F)=\mathbb{C}^{16} @>>> \mathrm{Ext}^{1}(A,F)=\mathbb{C}^{4}\\
@VVV @VVV @VVV\\
\mathrm{Ext}^{1}(B,B)=\mathbb{C}^{3} @>>> \mathrm{Ext}^{1}(F,B)=\mathbb{C}^{4} @>>> \mathrm{Ext}^{1}(A,B)=\mathbb{C}\\
@VVV @V0VV @V0VV\\
0 @>>> \mathrm{Ext}^{2}(F,A)=\mathbb{C}^{3} @>>> \mathrm{Ext}^{2}(A,A)=\mathbb{C}^{3}\\
@VVV @VVV @VVV\\
0 @>>> \mathrm{Ext}^{2}(F,F)=\mathbb{C}^{4} @>>> \mathrm{Ext}^{2}(A,F)=\mathbb{C}^{4}\\
@VVV @VVV @VVV\\
0 @>>> \mathrm{Ext}^{2}(F,B)=\mathbb{C} @>>> \mathrm{Ext}^{2}(A,B)=\mathbb{C}
\end{CD}$
\end{center}
\begin{proof}
We show that the diagram holds if and only if $F\in\mathbb{P}(\mathcal{N}_{H/\mathbf{M}_{1}}^{*})$. If the diagram holds, then $\theta_{F}\neq0$. We can find $\zeta\in\mathrm{Ext}^{1}(F,F)$ such that $e=l[1]\circ\zeta\circ k$. Now we have $f\circ e[-1]=f\circ l\circ\zeta[-1]\circ k[-1]=0$ because $f\circ l=0$. This means $f:B\longrightarrow A[1]$ factors through $h:B\longrightarrow E$, i.e. $f=x\circ h$ for some $x:E\longrightarrow A[1]$. On the other hand, from the diagram in Lemma \ref{zheteng} we see that $\mathrm{Ext}^{1}(E,E)\overset{j_{*}}{\longrightarrow}\mathrm{Ext}^{1}(E,A)$ is surjective, hence
$x:E\longrightarrow A[1]$ lifts to some $\xi:E\longrightarrow E[1]$. So we have $f=j[1]\circ\xi\circ h$ and $f$ is in the image of $\theta_{E}$. By Proposition \ref{KS}, this means $f$ is in $\mathbb{P}(\mathcal{N}_{H/\mathbf{M}_{1}}^{*})$. Conversely, if $f$ is in $\mathbb{P}(\mathcal{N}_{H/\mathbf{M}_{1}}^{*})$, then we can write $f=j[1]\circ\xi\circ h$ for some nontrivial $\xi:E\longrightarrow E[1]$. Then $f[1]\circ e=j[2]\circ\xi[1]\circ h[1]\circ e=0$ because $h[1]\circ e=0$. This means $e:A\longrightarrow B[1]$ factors through $l[1]:F[1]\longrightarrow B[1]$, i.e. $e=l[1]\circ z$ for some $z:A\longrightarrow F[1]$. On the other hand, $\mathrm{Ext}^{1}(F,F)\overset{k^{*}}{\longrightarrow}\mathrm{Ext}^{1}(A,F)$ is surjective because its cokernel $\mathrm{Ext}^{2}(B,F)=0$. This implies that $z=\zeta\circ k$ for some $\zeta:E\longrightarrow E[1]$. So we have $e=l[1]\circ\zeta\circ k$ and $e$ is in the image of $\theta_{F}$. Therefore $\theta_{F}\neq0$. By Corollary $4.7$, the kernel of $\theta_{F}$ is $T_{\mathbf{P},F}$, which is $15$-dimensional since $\mathbf{P}$ is a $\mathbb{P}^{9}$-bundle over $\mathbb{P}^{3}\times(\mathbb{P}^{3})^{*}$. Hence $\mathrm{Ext}^{1}(F,F)=\mathbb{C}^{16}$. The rest of the diagram will follow automatically due to exactness.
\end{proof}
\begin{propo}
If $F\in \mathbb{P}(\mathscr{E}xt^{1}_{\pi_{H}}(\mathcal{G}_{H},\mathcal{F}_{H})^{*})\setminus\mathbb{P}(\mathcal{N}_{H/\mathbf{M}_{1}}^{*})$, then we have the following commutative diagram with exact rows and columns. All boundary homomorphisms are $0$ except at $\mathrm{Ext}^{1}(B,A)$, where the two homomophisms $\mathrm{Ext}^{1}(F,A)\longleftarrow\mathrm{Ext}^{1}(B,A)\longrightarrow\mathrm{Ext}^{1}(B,F)$ have a same $1$-dimensional kernel $\mathbb{C}f$.
\end{propo}
\begin{center}
$\begin{CD}
\mathrm{Ext}^{1}(B,A)=\mathbb{C}^{10} @>>> \mathrm{Ext}^{1}(F,A)=\mathbb{C}^{12} @>>> \mathrm{Ext}^{1}(A,A)=\mathbb{C}^{3}\\
@VVV @VVV @VVV\\
\mathrm{Ext}^{1}(B,F)=\mathbb{C}^{12} @>>> \mathrm{Ext}^{1}(F,F)=\mathbb{C}^{15} @>>> \mathrm{Ext}^{1}(A,F)=\mathbb{C}^{3}\\
@VVV @VVV @V0VV\\
\mathrm{Ext}^{1}(B,B)=\mathbb{C}^{3} @>>> \mathrm{Ext}^{1}(F,B)=\mathbb{C}^{4} @>>> \mathrm{Ext}^{1}(A,B)=\mathbb{C}\\
@VVV @VVV @VVV\\
0 @>>> \mathrm{Ext}^{2}(F,A)=\mathbb{C}^{3} @>>> \mathrm{Ext}^{2}(A,A)=\mathbb{C}^{3}\\
@VVV @VVV @VVV\\
0 @>>> \mathrm{Ext}^{2}(F,F)=\mathbb{C}^{3} @>>> \mathrm{Ext}^{2}(A,F)=\mathbb{C}^{3}\\
@VVV @VVV @VVV\\
0 @>>> \mathrm{Ext}^{2}(F,B)=\mathbb{C} @>>> \mathrm{Ext}^{2}(A,B)=\mathbb{C}\\
\end{CD}$
\end{center}
\begin{proof}
By the proof of previous proposition, we know that $\theta_{F}=0$ since $F$ is not in $\mathbb{P}(\mathcal{N}_{H/\mathbf{M}_{1}}^{*})$. Therefore $\mathrm{Ext}^{1}(F,F)=\mathbb{C}^{15}$. By Lemma \ref{diao}, we know $\mathrm{Ext}^{1}(A,B)=\mathbb{C}$, since $F$ is mapped into $H$ under the bundle projection $\mathbf{P}\longrightarrow\mathbb{P}^{3}\times(\mathbb{P}^{3})^{*}$. The rest of the diagram then follows automatically due to exactness.
\end{proof}
\begin{propo}
If $F\in \mathbf{P}\setminus\mathbb{P}(\mathscr{E}xt^{1}_{\pi_{H}}(\mathcal{G}_{H},\mathcal{F}_{H})^{*})$, then we have the following commutative diagram with exact rows and columns. All boundary homomorphisms are $0$ except at $\mathrm{Ext}^{1}(B,A)$, where the two homomophisms $\mathrm{Ext}^{1}(F,A)\longleftarrow\mathrm{Ext}^{1}(B,A)\longrightarrow\mathrm{Ext}^{1}(B,F)$ have a same $1$-dimensional kernel $\mathbb{C}f$.
\end{propo}
\begin{center}
$\begin{CD}
\mathrm{Ext}^{1}(B,A)=\mathbb{C}^{10} @>>> \mathrm{Ext}^{1}(F,A)=\mathbb{C}^{12} @>>> \mathrm{Ext}^{1}(A,A)=\mathbb{C}^{3}\\
@VVV @VVV @VVV\\
\mathrm{Ext}^{1}(B,F)=\mathbb{C}^{12} @>>> \mathrm{Ext}^{1}(F,F)=\mathbb{C}^{15} @>>> \mathrm{Ext}^{1}(A,F)=\mathbb{C}^{3}\\
@VVV @VVV @VVV\\
\mathrm{Ext}^{1}(B,B)=\mathbb{C}^{3} @>>> \mathrm{Ext}^{1}(F,B)=\mathbb{C}^{3} @>>> \mathrm{Ext}^{1}(A,B)=0\\
@VVV @V0VV @VVV\\
0 @>>> \mathrm{Ext}^{2}(F,A)=\mathbb{C}^{3} @>>> \mathrm{Ext}^{2}(A,A)=\mathbb{C}^{3}\\
@VVV @VVV @VVV\\
0 @>>> \mathrm{Ext}^{2}(F,F)=\mathbb{C}^{4} @>>> \mathrm{Ext}^{2}(A,F)=\mathbb{C}^{4}\\
@VVV @VVV @VVV\\
0 @>>> \mathrm{Ext}^{2}(F,B)=\mathbb{C} @>>> \mathrm{Ext}^{2}(A,B)=\mathbb{C}\\
\end{CD}$
\end{center}
\begin{proof}
Since $F$ is not in $\mathbb{P}(\mathcal{N}_{H/\mathbf{M}_{1}}^{*})$, we have $\theta_{F}=0$ and $\mathrm{Ext}^{1}(F,F)=\mathbb{C}^{15}$. By Lemma \ref{diao}, we know $\mathrm{Ext}^{1}(A,B)=0$ since $F$ is mapped outside $H$ under the bundle projection $\mathbf{P}\longrightarrow\mathbb{P}^{3}\times(\mathbb{P}^{3})^{*}$. The rest of the diagram then follows automatically due to exactness.
\end{proof}
\begin{remark}
From the above propositions, we can see that for $F\in \mathbf{P}\setminus\mathbb{P}(\mathcal{N}_{H/\mathbf{M}_{1}}^{*})$, $\mathbf{P}$ is smooth at $F$ and $\mathrm{dim} T_{\mathbf{P},F}=\mathrm{dim}T_{\mathbf{M}_{2},F}=15$. By Proposition \ref{zhongyao} (2), $T_{\varphi_{F},F}$ is injective. This implies $\varphi_{F}$ is an isomorphism at $F$ and $\mathbf{M}_{2}$ is smooth at $F$.
\label{fadian}
\end{remark}
\smallskip

\noindent\textbf{Elementary modification.} In this subsection, we construct a flat family of $\lambda_{2}$-stable objects on the blow-up of $\mathbf{M}_{1}$ along $H$. The key is to perform a so-called elementary modification on the pullback of universal family of $\lambda_{1}$-stable objects along the exceptional divisor with respect to the extension $(1)$ in Proposition \ref{dadiao}.

Let us first introduce some notations: denote the blow-up of $\mathbf{M}_{1}$ along $H$ by $\mathbf{B}$, the blow-up morphism $\mathbf{B}\times\mathbb{P}^{3}\longrightarrow \mathbf{M}_{1}\times\mathbb{P}^{3}$ by $b$ and its restriction to the exceptional divisor $\mathbb{P}(\mathcal{N}_{H/\mathbf{M}_{1}}^{*})\times\mathbb{P}^{3}\longrightarrow H\times\mathbb{P}^{3}$ by $b_{H}$. Denote the universal family of $\lambda_{1}$-stable objects on $\mathbf{M}_{1}\times\mathbb{P}^{3}$ by $\mathcal{U}_{1}$, then $\mathcal{U}_{1}|_{H\times\mathbb{P}^{3}}$ and $\mathcal{U}_{E}$ both induce the embedding $\varphi_{E}:H\longrightarrow\mathbf{M}_{1}$, so they differ from each other by tensoring a pullback of a line bundle from $H$ via projection. Assume $\mathcal{U}_{1}|_{H\times\mathbb{P}^{3}}=\mathcal{U}_{E}\otimes\pi_{H}^{*}\mathcal{L}'$ for some line bundle $\mathcal{L}'$ on $H$. Consider the composition of the restriction map and the pullback of surjection in (1) by $b_{H}$:
\begin{equation}
b^{*}\mathcal{U}_{1}\twoheadrightarrow b^{*}\mathcal{U}_{1}|_{\mathbb{P}(\mathcal{N}_{H/\mathbf{M}_{1}}^{*})\times\mathbb{P}^{3}}=b_{H}^{*}\mathcal{U}_{E}\otimes b_{H}^{*}\pi_{H}^{*}\mathcal{L}'\twoheadrightarrow b_{H}^{*}\mathcal{F}_{H}\otimes b_{H}^{*}\pi_{H}^{*}\mathcal{L}'\nonumber
\end{equation} 
Denote the kernel of this composition by $\mathcal{K}$ then we have:
\begin{propo}
The sheaf $\mathcal{K}$ is a flat family of $\lambda_{2}$-stable objects. 
\end{propo}
\begin{proof}
$\mathcal{K}$ is a flat family of $\lambda_{2}$-stable objects outside the exceptional divisor because it is identical to $\mathcal{U}_{1}$. If we restrict the exact sequence $0\longrightarrow\mathcal{K}\longrightarrow b^{*}\mathcal{U}_{1}\longrightarrow b_{H}^{*}\mathcal{F}_{H}\otimes b_{H}^{*}\pi_{H}^{*}\mathcal{L}'\longrightarrow0$ to the exceptional divisor $\mathbb{P}(\mathcal{N}_{H/\mathbf{M}_{1}}^{*})\times\mathbb{P}^{3}$, we will get
\begin{align*}
0\longrightarrow\mathscr{T}or^{1}(b_{H}^{*}\mathcal{F}_{H}\otimes b_{H}^{*}\pi_{H}^{*}\mathcal{L}',\mathcal{O}_{\mathbb{P}(\mathcal{N}_{H/\mathbf{M}_{1}}^{*})\times\mathbb{P}^{3}})\longrightarrow\mathcal{K}|_{\mathbb{P}(\mathcal{N}_{H/\mathbf{M}_{1}}^{*})\times\mathbb{P}^{3}}&\longrightarrow\\b_{H}^{*}\mathcal{U}_{E}\otimes b_{H}^{*}\pi_{H}^{*}\mathcal{L}'\longrightarrow b_{H}^{*}\mathcal{F}_{H}\otimes b_{H}^{*}\pi_{H}^{*}\mathcal{L}'&\longrightarrow0\nonumber
\end{align*}
On the other hand, tensoring $b_{H}^{*}\mathcal{F}_{H}\otimes b_{H}^{*}\pi_{H}^{*}\mathcal{L}'$ to the exact sequence $0\longrightarrow\mathcal{I}_{\mathbb{P}(\mathcal{N}_{H/\mathbf{M}_{1}}^{*})\times\mathbb{P}^{3}}\longrightarrow\mathcal{O}\longrightarrow\mathcal{O}_{\mathbb{P}(\mathcal{N}_{H/\mathbf{M}_{1}}^{*})\times\mathbb{P}^{3}}\longrightarrow0$, we have
\begin{align*}
0\longrightarrow\mathscr{T}or^{1}(b_{H}^{*}\mathcal{F}_{H}\otimes b_{H}^{*}\pi_{H}^{*}\mathcal{L}',\mathcal{O}_{\mathbb{P}(\mathcal{N}_{H/\mathbf{M}_{1}}^{*})\times\mathbb{P}^{3}})\overset{=}{\longrightarrow} b_{H}^{*}\mathcal{F}_{H}\otimes b_{H}^{*}\pi_{H}^{*}\mathcal{L}'\otimes\mathcal{I}_{\mathbb{P}(\mathcal{N}_{H/\mathbf{M}_{1}}^{*})\times\mathbb{P}^{3}}\\\overset{0}{\longrightarrow} b_{H}^{*}\mathcal{F}_{H}\otimes b_{H}^{*}\pi_{H}^{*}\mathcal{L}'\overset{=}{\longrightarrow}  b_{H}^{*}\mathcal{F}_{H}\otimes b_{H}^{*}\pi_{H}^{*}\mathcal{L}'\longrightarrow0.\nonumber
\end{align*}
Hence 
\begin{align*}
\mathscr{T}or^{1}(b_{H}^{*}\mathcal{F}_{H}\otimes b_{H}^{*}\pi_{H}^{*}\mathcal{L}',\mathcal{O}_{\mathbb{P}(\mathcal{N}_{H/\mathbf{M}_{1}}^{*})\times\mathbb{P}^{3}})&=b_{H}^{*}\mathcal{F}_{H}\otimes b_{H}^{*}\pi_{H}^{*}\mathcal{L}'\otimes\mathcal{I}_{\mathbb{P}(\mathcal{N}_{H/\mathbf{M}_{1}}^{*})\times\mathbb{P}^{3}}\\&= b_{H}^{*}\mathcal{F}_{H}\otimes b_{H}^{*}\pi_{H}^{*}\mathcal{L}'\otimes\mathcal{N}_{\mathbb{P}(\mathcal{N}_{H/\mathbf{M}_{1}}^{*})\times\mathbb{P}^{3}}^{*}.
\end{align*}
Also notice that the kernel of
\begin{equation*}
b_{H}^{*}\mathcal{U}_{E}\otimes b_{H}^{*}\pi_{H}^{*}\mathcal{L}'\longrightarrow b_{H}^{*}\mathcal{F}_{H}\otimes b_{H}^{*}\pi_{H}^{*}\mathcal{L}'
\end{equation*} is $b_{H}^{*}\mathcal{G}_{H}\otimes b_{H}^{*}\pi_{H}^{*}\mathcal{L}\otimes b_{H}^{*}\pi_{H}^{*}\mathcal{L}'$, so $\mathcal{K}|_{\mathbb{P}(\mathcal{N}_{H/\mathbf{M}_{1}}^{*})\times\mathbb{P}^{3}}$ satisfies
\begin{align}
0\longrightarrow  b_{H}^{*}\mathcal{F}_{H}\otimes b_{H}^{*}\pi_{H}^{*}\mathcal{L}'\otimes\mathcal{N}_{\mathbb{P}(\mathcal{N}_{H/\mathbf{M}_{1}}^{*})\times\mathbb{P}^{3}}^{*}\longrightarrow\mathcal{K}|_{\mathbb{P}(\mathcal{N}_{H/\mathbf{M}_{1}}^{*})\times\mathbb{P}^{3}}&\longrightarrow\nonumber\\ b_{H}^{*}\mathcal{G}_{H}\otimes b_{H}^{*}\pi_{H}^{*}\mathcal{L}\otimes b_{H}^{*}\pi_{H}^{*}\mathcal{L}'&\longrightarrow0.
\end{align}
This means on each fiber $x\times\mathbb{P}^{3}$, the restriction $\mathcal{K}_{x}$ is an extension of $B$ by $A$. In particular $\mathcal{K}_{x}$ has the same Chern character as other fibers, therefore $\mathcal{K}$ is flat since $\mathbf{B}$ is smooth. To prove it is a family of $\lambda_{2}$-stable objects, we need to show $\mathcal{K}_{x}$ is a nontrivial extension of $B$ by $A$. Actually since $x\in\mathbb{P}(\mathcal{N}_{H/\mathbf{M}_{1}}^{*})$ represents a nonzero normal direction of $H$ in $\mathbf{M}_{1}$, we expect $\mathcal{K}_{x}$ to be $\theta_{E}(\mathrm{KS}(x))$ in $\mathrm{Ext}^{1}(B,A)$. This is indeed the case because  $\mathcal{K}|_{\mathbb{P}(\mathcal{N}_{H/\mathbf{M}_{1}}^{*})\times\mathbb{P}^{3}}$ can be interpreted in the following way: First we use the injection
\begin{equation*}
b_{H}^{*}\mathcal{G}_{H}\otimes b_{H}^{*}\pi_{H}^{*}\mathcal{L}\otimes b_{H}^{*}\pi_{H}^{*}\mathcal{L}'\longrightarrow b_{H}^{*}\mathcal{U}_{E}\otimes b_{H}^{*}\pi_{H}^{*}\mathcal{L}'
\end{equation*}
to pull back the exact sequence
\begin{equation*}
0\longrightarrow b^{*}\mathcal{U}_{1}\otimes\mathcal{I}_{\mathbb{P}(\mathcal{N}_{H/\mathbf{M}_{1}}^{*})\times\mathbb{P}^{3}}\longrightarrow b^{*}\mathcal{U}_{1}\longrightarrow b_{H}^{*}\mathcal{U}_{E}\otimes b_{H}^{*}\pi_{H}^{*}\mathcal{L}'\longrightarrow0,
\end{equation*}we get
\begin{equation*}
0\longrightarrow b^{*}\mathcal{U}_{1}\otimes\mathcal{I}_{\mathbb{P}(\mathcal{N}_{H/\mathbf{M}_{1}}^{*})\times\mathbb{P}^{3}}\longrightarrow\mathcal{K}\longrightarrow b_{H}^{*}\mathcal{G}_{H}\otimes b_{H}^{*}\pi_{H}^{*}\mathcal{L}\otimes b_{H}^{*}\pi_{H}^{*}\mathcal{L}'\longrightarrow0.
\end{equation*}Then we push out the resulting exact sequence using the surjection 
\begin{equation*}
b^{*}\mathcal{U}_{1}\otimes\mathcal{I}_{\mathbb{P}(\mathcal{N}_{H/\mathbf{M}_{1}}^{*})\times\mathbb{P}^{3}}\longrightarrow b_{H}^{*}\mathcal{F}_{H}\otimes b_{H}^{*}\pi_{H}^{*}\mathcal{L}'\otimes\mathcal{I}_{\mathbb{P}(\mathcal{N}_{H/\mathbf{M}_{1}}^{*})\times\mathbb{P}^{3}}=b_{H}^{*}\mathcal{F}_{H}\otimes b_{H}^{*}\pi_{H}^{*}\mathcal{L}'\otimes\mathcal{N}_{\mathbb{P}(\mathcal{N}_{H/\mathbf{M}_{1}}^{*})\times\mathbb{P}^{3}}^{*},
\end{equation*}
we will get $(3)$. On a fiber $x\times\mathbb{P}^{3}$, this means first we take an extension
\begin{equation*}
0\longrightarrow E\longrightarrow G\longrightarrow E\longrightarrow0
\end{equation*}representing $x\in\mathrm{Ext}^{1}(E,E)$, then do a pullback using $B\longrightarrow E$ followed by a pushout using $E\longrightarrow A$. The resulting extension
\begin{equation*}
0\longrightarrow A\longrightarrow\mathcal{K}_{x}\longrightarrow B\longrightarrow0
\end{equation*}is exactly $\theta_{E}(\mathrm{KS}(x))$. This shows that $\mathcal{K}$ is a flat family of $\lambda_{2}$-stable objects.
\end{proof}
If we denote the induced morphism of $\mathcal{K}$ by $\delta:\mathbf{B}\longrightarrow \mathbf{M}_{2}$, then
\begin{propo}
(1) The induced morphism $\delta$ is an isomorphism outside $\mathbb{P}(\mathcal{N}_{H/\mathbf{M}_{1}}^{*})$, and the restriction $\delta|_{\mathbb{P}(\mathcal{N}_{H/\mathbf{M}_{1}}^{*})}$ coincides with $\varphi_{F}|_{\mathbb{P}(\mathcal{N}_{H/\mathbf{M}_{1}}^{*})}$;

(2) The induced morphism $\delta$ is injective on the level of sets and Zariski tangent spaces.
\label{ruyao}
\end{propo}
\begin{proof}
$\delta$ is an isomorphism outside $\mathbb{P}(\mathcal{N}_{H/\mathbf{M}_{1}}^{*})$ because $\mathcal{K}$ is the same with $\mathcal{U}_{1}$. On the other hand, under the identification 
\begin{align*}
&\mathrm{Ext}^{1}\left(b_{H}^{*}\mathcal{G}_{H}\otimes b_{H}^{*}\pi_{H}^{*}\mathcal{L}\otimes b_{H}^{*}\pi_{H}^{*}\mathcal{L}',b_{H}^{*}\mathcal{F}_{H}\otimes b_{H}^{*}\pi_{H}^{*}\mathcal{L}'\otimes\mathcal{N}_{\mathbb{P}(\mathcal{N}_{H/\mathbf{M}_{1}}^{*})\times\mathbb{P}^{3}}^{*}\right)\\&=\mathrm{Ext}^{1}\left(b_{H}^{*}\mathcal{G}_{H}\otimes b_{H}^{*}\pi_{H}^{*}\mathcal{L},b_{H}^{*}\mathcal{F}_{H}\otimes\mathcal{N}_{\mathbb{P}(\mathcal{N}_{H/\mathbf{M}_{1}}^{*})\times\mathbb{P}^{3}}^{*}\right)\nonumber\\&=\mathrm{H}^{0}\left(\mathbb{P}(\mathcal{N}_{H/\mathbf{M}_{1}}^{*}),\mathscr{E}xt_{\pi_{\mathbb{P}(\mathcal{N}_{H/\mathbf{M}_{1}}^{*})}}^{1}\left(b_{H}^{*}\mathcal{G}_{H}\otimes b_{H}^{*}\pi_{H}^{*}\mathcal{L},b_{H}^{*}\mathcal{F}_{H}\otimes\pi_{\mathbb{P}(\mathcal{N}_{H/\mathbf{M}_{1}}^{*})}^{*}\mathcal{O}_{\mathbb{P}(\mathcal{N}_{H/\mathbf{M}_{1}}^{*})}(1)\right)\right)\\&=\mathrm{H}^{0}\left(H,\mathscr{E}xt_{\pi_{H}}^{1}\left(\mathcal{G}_{H}\otimes\pi_{H}^{*}\mathcal{L},\mathcal{F}_{H}\right)\otimes\mathcal{N}_{H/\mathbf{M}_{1}}^{*}\right)\\&=\mathrm{Hom}\left(\mathcal{N}_{H/\mathbf{M}_{1}},\mathscr{E}xt_{\pi_{H}}^{1}\left(\mathcal{G}_{H}\otimes\pi_{H}^{*}\mathcal{L},\mathcal{F}_{H}\right)\right),
\end{align*}
\noindent the extension $(3)$ corresponds to the injection $i$ from $\mathcal{N}_{H/\mathbf{M}_{1}}$ to $\mathscr{E}xt_{\pi_{H}}^{1}(\mathcal{G}_{H}\otimes\pi_{H}^{*}\mathcal{L},\mathcal{F}_{H})$ constructed in Proposition \ref{duodiao} via the Kodaira-Spencer map. Similarly in Proposition \ref{juru}, the extension $(2)$ corresponds to the identity $id$ in $\mathrm{Hom}(\mathscr{E}xt^{1}_{\pi}(\mathcal{G},\mathcal{F}),\mathscr{E}xt^{1}_{\pi}(\mathcal{G},\mathcal{F}))=\mathrm{Ext}^{1}(h^{*}\mathcal{G},h^{*}\mathcal{F}\otimes\pi_{\mathbf{P}}\mathcal{O}_{\mathbf{P}}(1))$. Notice that $i$ is the restriction of $id$ to $\mathcal{N}_{H/\mathbf{M}_{1}}$, this means $(3)$ is a restriction of $(2)$ to $\mathbb{P}(\mathcal{N}_{H/\mathbf{M}_{1}}^{*})\times\mathbb{P}^{3}$ up to tensoring a pullback of some line bundle on $\mathbb{P}(\mathcal{N}_{H/\mathbf{M}_{1}}^{*})$. Therefore $\delta|_{\mathbb{P}(\mathcal{N}_{H/\mathbf{M}_{1}}^{*})\times\mathbb{P}^{3}}=\varphi_{F}|_{\mathbb{P}(\mathcal{N}_{H/\mathbf{M}_{1}}^{*})\times\mathbb{P}^{3}}$. In particular, $\delta|_{\mathbb{P}(\mathcal{N}_{H/\mathbf{M}_{1}}^{*})\times\mathbb{P}^{3}}$ is injective on the level of Zariski tangent spaces since $\varphi_{F}$ is. To show $\delta$ is injective on the level of Zariski tangent spaces, it only remains to show that the normal direction $v_{x}$ of $\mathbb{P}(\mathcal{N}_{H/\mathbf{M}_{1}}^{*})$ in $\mathbf{B}$ at a point $x\in\mathbb{P}(\mathcal{N}_{H/\mathbf{M}_{1}}^{*})$ is not sent to the image of $T_{\mathbb{P}(\mathcal{N}_{H/\mathbf{M}_{1}}^{*}),x}$ under $T_{\delta,x}$. If it were so, we suppose $\xi:\mathrm{Spec}\mathbb{C}[\varepsilon]/(\varepsilon^{2})\longrightarrow\mathbf{B}$ represents $v_{x}$. Notice that we have a pullback diagram
\begin{center}
$\begin{CD}
\mathbb{P}(\mathcal{N}_{H/\mathbf{M}_{1}}^{*}) @>>> \mathbf{P}\\
@VVV @V\varphi_{F}VV\\
\mathbf{B} @>\delta>> \mathbf{M}_{2}
\end{CD}$
\end{center} 
since $\delta(\mathbf{B})\cap\varphi_{F}(\mathbf{P})=\delta(\mathbb{P}(\mathcal{N}_{H/\mathbf{M}_{1}}^{*}))$. Because $T_{\delta,x}(T_{\mathbb{P}(\mathcal{N}_{H/\mathbf{M}_{1}}^{*}),x})$ is contained in $T_{\varphi_{F},x}$, we can lift $\delta\circ\xi$ to $\xi':\mathrm{Spec}\mathbb{C}[\varepsilon]/(\varepsilon^{2})\longrightarrow \mathbf{P}$ that makes the pullback diagram above commutative, hence $\xi$ factors through $\mathbb{P}(\mathcal{N}_{H/\mathbf{M}_{1}}^{*})$. This implies $v_{x}$ is in $T_{\mathbb{P}(\mathcal{N}_{H/\mathbf{M}_{1}}^{*}),x}$, which is a contradiction.
\end{proof}
\begin{remark}
(1) The last argument also shows that the normal direction $v_{x}$ is not mapped to the image of $T_{\mathbf{P},\mathcal{K}_{x}}$ under  $T_{\varphi_{F},F}$. By Corollary $4.7$, $T_{\varphi_{F},F}(T_{\mathbf{P},\mathcal{K}_{x}})$ is the kernel of $\theta_{F}$, so we must have $\theta_{F}(v_{x})\neq0$;

(2) Since $T_{\varphi_{F},F}(T_{\mathbf{P},F})=\mathbb{C}^{15}$ and $T_{\delta,F}(T_{\mathbf{B},F})=\mathbb{C}^{12}$, the pullback diagram in the above proof also implies $T_{\varphi_{F},F}(T_{\mathbf{P},F})\cap T_{\delta,F}(T_{\mathbf{B},F})=T_{\delta,F}(T_{\mathbb{P}(\mathcal{N}_{H/\mathbf{M}_{1}}^{*}),F})=\mathbb{C}^{11}$.
\label{mafuyu}
\end{remark}
\smallskip

\noindent\textbf{Obstruction computation.} In this subsection, we study the deformation theory of complexes on the intersection of the two irreducible components of $\mathbf{M}_{2}$. We give explicit local equations defining $\mathbf{M}_{2}$ at a point in the intersection. In particular, this will imply the two irreducible components of $\mathbf{M}_{2}$ intersect transversely.

Recall that we have constructed two morphisms $\delta:\mathbf{B}\longrightarrow \mathbf{M}_{2}$ and $\varphi_{F}:\mathbf{P}\longrightarrow \mathbf{M}_{2}$, both of them are injective on the level of sets and Zariski tangent spaces. By the definition of a simple wall-crossing, any $\lambda_{2}$-stable object has to lie in the image of one of the two morphisms. Thus $\mathbf{M}_{2}$ has two irreducible components corresponding to the image of $\delta$ and $\varphi_{F}$. The intersection of the two components is the image of the exceptional divisor $\mathbb{P}(\mathcal{N}_{H/\mathbf{M}_{1}}^{*})$ by Proposition \ref{ruyao}. Outside the intersection of the two components, $\mathbf{M}_{2}$ is smooth by Remark \ref{fadian} and Remark \ref{mafuyu} (1). To study the singularity of $\mathbf{M}_{2}$, we fix an $\lambda_{2}$-semistable object $F$ in $\mathbb{P}(\mathcal{N}_{H/\mathbf{M}_{1}}^{*})$, then we have
\begin{propo}
The tangent vectors of $\mathbf{M}_{2}$ at $F$ in the subspaces $T_{\varphi_{F},F}(T_{\mathbf{P},F})$ and $T_{\delta,F}(T_{\mathbf{B},F})$ correspond to versal deformations of $F$.
\label{mua}
\end{propo}
\begin{proof}
Suppose a Zariski tangent vector of $\mathbf{M}_{2}$ at $F$ in $T_{\varphi_{F},F}(T_{\mathbf{P},F})$ is represented by a morphism $\eta:\mathrm{Spec}\mathbb{C}[\varepsilon]/(\varepsilon^{2})\longrightarrow\mathbf{M}_{2}$, then $\eta$ factors through $\varphi_{F}:\mathbf{P}\longrightarrow\mathbf{M}_{2}$:
\begin{displaymath}
\xymatrix {\mathrm{Spec}\mathbb{C}[\varepsilon]/(\varepsilon^{2})\ar[r]\ar[d]^{\eta'}\ar[dr]^{\eta}&\mathrm{Spec}S\ar[dl]^{\xi} \\\mathbf{P}\ar[r]^{\varphi_{F}}&\mathbf{M}_{2}
}
\end{displaymath}
If $S$ is a finite dimensional local Artin $\mathbb{C}$-algebra with a local surjection $S\longrightarrow\mathbb{C}[\varepsilon]/(\varepsilon^{2})$, then we can lift $\eta'$ to $\xi:\mathrm{Spec}S\longrightarrow\mathbf{P}$ since $\mathbf{P}$ is smooth. By composing $\xi$ with $\varphi_{F}$, we get a lift of $\eta$. Hence $\eta$ corresponds to a versal deformation. A similar argument works for tangent vectors in $T_{\delta,F}(T_{\mathbf{B},F})$.
\end{proof} 
In order to show $T_{\varphi_{F},F}(T_{\mathbf{P},F})$ and $T_{\delta,F}(T_{\mathbf{B},F})$ are all the versal deformations of $F$, we study the quadratic part of the Kuranishi map $\kappa_{2}:T_{\mathbf{M}_{2},F}\cong\mathrm{Ext}^{1}(F,F)\longrightarrow\mathrm{Ext}^{2}(F,F)$. First we give a decomposition of $T_{\mathbf{M}_{2},F}\cong\mathrm{Ext}^{1}(F,F)$ with respect to some geometric structures. In the blow-up $\mathbf{B}$, we have $T_{\mathbf{B},F}=N_{\mathbb{P}(\mathcal{N}_{H/\mathbf{M}_{1}}^{*})/\mathbf{B},F}\oplus T_{\mathbb{P}(\mathcal{N}_{H/\mathbf{M}_{1}}^{*}),F}$ and $N_{\mathbb{P}(\mathcal{N}_{H/\mathbf{M}_{1}}^{*})/\mathbf{B},F}$ is 1-dimensional. Suppose it is generated by a vector $v_{F}$, then we have
\begin{propo}
The Zariski tangent space $T_{\mathbf{M}_{2},F}\cong\mathrm{Ext}^{1}(F,F)$ has the following decomposition
\begin{equation}
T_{\mathbf{M}_{2},F}=\mathbb{C}v_{F}\oplus T_{\mathbb{P}(N_{H/\mathbf{M}_{1},E}^{*}),F}\oplus N_{\mathbb{P}(N_{H/\mathbf{M}_{1},E}^{*})/\mathbb{P}(\mathrm{Ext}^{1}(B,A)^{*}),F}\oplus T_{H,E}\oplus N_{H/\mathbb{P}^{3}\times(\mathbb{P}^{3})^{*},E}.
\end{equation}In this decomposition,
\begin{align*}
T_{\delta,F}(T_{\mathbf{B},F})&=\mathbb{C}v_{F}\oplus T_{\mathbb{P}(N_{H/\mathbf{M}_{1},E}^{*}),F}\oplus T_{H,E}\\ T_{\varphi_{F},F}(T_{\mathbf{P},F})&=T_{\mathbb{P}(N_{H/\mathbf{M}_{1},E}^{*}),F}\oplus N_{\mathbb{P}(N_{H/\mathbf{M}_{1},E}^{*})/\mathbb{P}(\mathrm{Ext}^{1}(B,A)^{*}),F}\oplus T_{H,E}\oplus N_{H/\mathbb{P}^{3}\times(\mathbb{P}^{3})^{*},E}
\end{align*}
\label{miao}
\end{propo}
\begin{proof}
By Remark \ref{mafuyu} (1), $\theta_{F}(v_{F})\neq0$, hence we can decompose $\mathrm{Ext}^{1}(F,F)=\mathbb{C}v_{F}\oplus T_{\mathbf{P},F}$ because the kernel of $\theta_{F}$ is $T_{\mathbf{P},F}$. On the other hand, $\mathbf{P}=\mathbb{P}(\mathscr{E}xt^{1}_{\pi}(\mathcal{G},\mathcal{F})^{*})$ is a projective bundle over $\mathbb{P}^{3}\times(\mathbb{P}^{3})^{*}$, so we have $T_{\mathbf{P},F}=T_{\mathbb{P}({\mathrm{Ext}^{1}(B,A)^{*}),F}}\oplus T_{\mathbb{P}^{3}\times(\mathbb{P}^{3})^{*},(A,B)}$. To give further decomposition, denote $E$ the nontrivial extension of $A$ by $B$, we have that $\mathbb{P}(N_{H/\mathbf{M}_{1},E}^{*})$ is embedded in $\mathbb{P}(\mathrm{Ext}^{1}(B,A)^{*})$ via the Kodaira-Spencer map by Proposition \ref{KS}, so $T_{\mathbb{P}({\mathrm{Ext}^{1}(B,A)^{*}),F}}=T_{\mathbb{P}(N_{H/\mathbf{M}_{1},E}^{*}),F}\oplus N_{\mathbb{P}(N_{H/\mathbf{M}_{1},E}^{*})/\mathbb{P}(\mathrm{Ext}^{1}(B,A)^{*}),F}$. Also notice that the incidence hyperplane $H$ is embedded in $\mathbb{P}^{3}\times(\mathbb{P}^{3})^{*}$, so $T_{\mathbb{P}^{3}\times(\mathbb{P}^{3})^{*},(A,B)}=T_{H,E}\oplus N_{H/\mathbb{P}^{3}\times(\mathbb{P}^{3})^{*},E}$. By composing all the decompositions above, we have the proposition.
\end{proof}

The importance of this decomposition is that some of the summands have direct relations with the $\mathrm{Ext}^{2}$ groups in Lemma \ref{zheteng}, Proposition \ref{KS} and Proposition \ref{cao}, which becomes crucial later when we compute $\kappa_{2}$. Fix a nontrivial  $\zeta\in\mathrm{Ext}^{1}(F,F)$. Let $e:A\longrightarrow B[1]$ correspond to the nontrivial extension $E$ and $f:B\longrightarrow A[1]$ correspond to $F$, name the arrows $ A\overset{k}{\longrightarrow}F\overset{l}{\longrightarrow}B$. Then we have the following two lemmas:
\begin{lemma}
The normal space $N_{\mathbb{P}(N_{H/\mathbf{M}_{1},E}^{*})/\mathbb{P}(\mathrm{Ext}^{1}(B,A)^{*}),F}$ can be identified with $\mathrm{Ext}^{2}(A,A)$ under a canonical isomorphism. If $\zeta$ belongs to $N_{\mathbb{P}(N_{H/\mathbf{M}_{1},E}^{*})/\mathbb{P}(\mathrm{Ext}^{1}(B,A)^{*}),F}$ in (4), then $\zeta=k[1]\circ t\circ l$ for some $t\in \mathrm{Ext}^{1}(B,A)$ such that $t[1]\circ e$ is nonzero in $\mathrm{Ext}^{2}(A,A)$.
\label{kao}
\end{lemma}
\begin{proof}
By Lemma \ref{zheteng}, we know that the cokernel of $\theta_{E}:\mathrm{Ext}^{1}(E,E)\longrightarrow\mathrm{Ext}^{1}(B,A)$ is $\mathrm{Ext}^{2}(A,A)$. By Proposition \ref{KS}, we know that the Kodaira-Spencer map $\mathrm{KS}$ induces an isomorphism between the image of $\theta_{E}$ and $N_{H/\mathbf{M}_{1},E}$. On the other hand, $N_{\mathbb{P}(N_{H/\mathbf{M}_{1},E}^{*})/\mathbb{P}(\mathrm{Ext}^{1}(B,A)^{*}),F}$ is equal to the quotient $\mathrm{Ext}^{1}(B,A)/N_{H/\mathbf{M}_{1},E}$, so $N_{\mathbb{P}(N_{H/\mathbf{M}_{1},E}^{*})/\mathbb{P}(\mathrm{Ext}^{1}(B,A)^{*}),F}\cong\mathrm{Ext}^{2}(A,A)$. To prove the second statement, we look at the square 
\begin{center}
$\begin{CD}
\mathrm{Ext}^{1}(B,A) @>l^{*}>> \mathrm{Ext}^{1}(F,A)\\
@Vk[1]_{*}VV @Vk[1]_{*}VV \\
\mathrm{Ext}^{1}(B,F) @>l^{*}>> \mathrm{Ext}^{1}(F,F)
\end{CD}$
\end{center} in Proposition \ref{cao}. There is an injection $\mathrm{Ext}^{1}(B,A)/\mathbb{C}f\longrightarrow\mathrm{Ext}^{1}(F,F)$, which is the same as $T_{\mathbb{P}(\mathrm{Ext}^{1}(B,A)^{*}),F}\longrightarrow\mathrm{Ext}^{1}(F,F)$. Notice the fact that $N_{\mathbb{P}(N_{H/\mathbf{M}_{1},E}^{*})/\mathbb{P}(\mathrm{Ext}^{1}(B,A)^{*}),F}$ is contained in $T_{\mathbb{P}(\mathrm{Ext}^{1}(B,A)^{*}),F}$, $\zeta$ has to be in $T_{\mathbb{P}(\mathrm{Ext}^{1}(B,A)^{*}),F}$, this means 
$\zeta=k[1]\circ t\circ l$ for some $t\in\mathrm{Ext}^{1}(B,A)$. For $\zeta$ to be nontrivial and lying in $\mathrm{Ext}^{2}(A,A)$, $t$ has to be nonzero under the cokernel map $(-)[1]\circ e:\mathrm{Ext}^{1}(B,A)\longrightarrow\mathrm{Ext}^{2}(A,A)
$, so
 $t[1]\circ e\neq0$
\end{proof}
\begin{lemma}
The normal space $N_{H/\mathbb{P}^{3}\times(\mathbb{P}^{3})^{*},E}$ can be identified with $\mathrm{Ext}^{2}(A,B)$ under a canonical isomorphism. If $\zeta$ belongs to $N_{H/\mathbb{P}^{3}\times(\mathbb{P}^{3})^{*},E}$ in (4), then $\zeta$ can be completed to the following commutative diagram with $e[1]\circ t+r[1]\circ e\neq0$ in $\mathrm{Ext}^{2}(A,B)$:
\begin{center}
$\begin{CD}
A @>k>> F @>l>> B\\
@VtVV @V\zeta VV @VrVV\\
A[1] @>k[1]>> F[1] @>l[1]>> B[1]
\end{CD}$
\end{center}
\label{ri}
\end{lemma}
\begin{proof}
Recall that $K_{E}$ is the kernel of $\theta_{E}$, and by Proposition \ref{KS} it can be identified with $T_{H,E}$ via the Kodaira-Spencer map. From the diagram in Lemma \ref{zheteng}, we have an exact sequence

\medskip
$\begin{CD}
0\longrightarrow K_{E}\longrightarrow\mathrm{Ext}^{1}(A,A)\oplus\mathrm{Ext}^{1}(B,B) @>(e[1]\circ-)+(-[1]\circ e)>> \mathrm{Ext}^{2}(A,B)\longrightarrow0
\end{CD}$.
\medskip

\noindent On the other hand, we have the canonical normal sequence of $H$ embedded in $\mathbb{P}^{3}\times(\mathbb{P}^{3})^{*}$
\begin{equation*}
0\longrightarrow T_{H,E}\longrightarrow T_{\mathbb{P}^{3}\times(\mathbb{P}^{3})^{*},(A,B)}\longrightarrow N_{H/\mathbb{P}^{3}\times(\mathbb{P}^{3})^{*},E}\longrightarrow0,
\end{equation*}
since $\mathrm{Ext}^{1}(A,A)\oplus\mathrm{Ext}^{1}(B,B)$ can also be identified with $T_{\mathbb{P}^{3}\times(\mathbb{P}^{3})^{*},(A,B)}$ via the Kodaira-Spencer map, this induces a canonical isomorphism between $N_{H/\mathbb{P}^{3}\times(\mathbb{P}^{3})^{*},E}$ and $\mathrm{Ext}^{2}(A,B)$.

Notice that $N_{H/\mathbb{P}^{3}\times(\mathbb{P}^{3})^{*},E}$ is contained in $T_{\mathbf{P},F}$ and the latter is kernel of $\theta_{F}$. We have $\theta_{F}(\zeta)=0$. By using the universal property of triangles, $\zeta$ can be completed to a commutative diagram:
\begin{center}
$\begin{CD}
A @>k>> F @>l>> B\\
@VtVV @V\zeta VV @VrVV\\
A[1] @>k[1]>> F[1] @>l[1]>> B[1]
\end{CD}$.
\end{center}
Since $\zeta$ is nontrivial, $(t,r)$ has to be sent to a nonzero element in $\mathrm{Ext}^{2}(A,B)$ under the last map of the exact sequence above, therefore $e[1]\circ t+r[1]\circ e\neq0$.
\end{proof}

With respect to the decomposition $(4)$, we let \begin{equation}
\zeta=u_{1}v_{F}+w_{1}+u_{2}s_{1}+u_{3}s_{2}+u_{4}s_{3}+w_{2}+u_{5}s_{4},
\end{equation}where $w_{1}\in T_{\mathbb{P}(N_{H/\mathbf{M}_{1},E}^{*}),F}$, $\{s_{1},s_{2},s_{3}\}$ forms a basis of $N_{\mathbb{P}(N_{H/\mathbf{M}_{1},E}^{*})/\mathbb{P}(\mathrm{Ext}^{1}(B,A)^{*}),F}$, $w_{2}\in T_{H,E}$, $\{s_{4}\}$ is a basis of $N_{H/\mathbb{P}^{3}\times(\mathbb{P}^{3})^{*},E}$ and $u_{i}\in\mathbb{C}$ are coefficients. $(5)$ is inspired by the explicit basis chosen in the proof of [PS85, Lemma 6]. In the next theorem, we will see that the equations cutting out versal deformations by using (5) is the same as using Piene and Schlessinger's basis in the case of deformations of ideals.
\begin{propo} The quadratic part of Kuranishi map takes the following form with respect to (5)
\begin{equation*}
\kappa_{2}(\zeta)=\zeta\cup\zeta= \sum_{i=1}^{4}u_{1}u_{i+1}(v_{F}+s_{i})\cup(v_{F}+s_{i}),
\end{equation*}
where $\cup$ is the Yoneda pairing of extensions. $\{(v_{F}+s_{i})\cup(v_{F}+s_{i})|i=1,2,3,4\}$ forms a basis of the obstruction space $\mathrm{Ext}^{2}(F,F)$.
\label{gan}
\end{propo}
\begin{proof}
The equality $\kappa_{2}(\zeta)=\zeta\cup\zeta$ is known for complexes in \cite{Ina02}, \cite{Lie06} and \cite{KLS06}. The second equality is a straightforward computation. It only uses the fact that for any $v$ in $T_{\mathbf{B},F}$ or $T_{\mathbf{P},F}$, we have $v\cup v=0$ since $v$ is a versal deformation by Proposition \ref{mua}.

To prove the last statement, we first show that $\{(v_{F}+s_{i})\cup(v_{F}+s_{i})|i=1,2,3\}$ is linearly independent. If not, then a certain nontrivial linear combination $\sum_{i=1}^{3}a_{i}(v_{F}+s_{i})\cup(v_{F}+s_{i})=0$. We can rewrite it as $v_{F}[1]\circ s+ s[1]\circ v_{F}=0$, where $s=\sum_{i=1}^{3}a_{i}s_{i}$ is a nontrivial first deformation of $F$ in $N_{\mathbb{P}(N_{H,E}^{*})/\mathbb{P}(\mathrm{Ext}^{1}(B,A)^{*}),F}$. By Lemma \ref{kao}, we can write $s=k[1]\circ t\circ l$ for some $t\in \mathrm{Ext}^{1}(B,A)$ such that $t[1]\circ e$ is nonzero in $\mathrm{Ext}^{2}(A,A)$. Now 
\begin{align}
0&=\left(v_{F}[1]\circ s+s[1]\circ v_{F}\right)\circ k\nonumber\\&=v_{F}[1]\circ k[1]\circ t\circ l\circ k+k[2]\circ t[1]\circ l[1]\circ v_{F}\circ k.\nonumber
\end{align}
Since $l\circ k=0$ and $l[1]\circ v_{F}\circ k=\theta_{F}(v_{F})=e$, we have $k[2]\circ t[1]\circ e=0$. From the diagram in Proposition \ref{cao}, we know that $\mathrm{Ext}^{2}(A,A)\overset{k[2]_{*}}{\longrightarrow}\mathrm{Ext}^{2}(A,F)$ is an injection, hence $t[1]\circ e=0$, which is a contradiction.

It only remains to show that $(v_{F}+s_{4})\cup(v_{F}+s_{4})$ is not a linear combination of $\{(v_{F}+s_{i})\cup(v_{F}+s_{i})|i=1,2,3\}$. For this we will show for $i=1,2,3$
\begin{align*}
l[2]\circ\left((v_{F}+s_{i})\cup(v_{F}+s_{i})\right)&=0,\\ l[2]\circ\left((v_{F}+s_{4})\cup(v_{F}+s_{4})\right)&\neq0.
\end{align*}
By Lemma \ref{kao}, we can assume $s_{i}=k[1]\circ t_{i}\circ l$ for some $t_{i}\in\mathrm{Ext}^{1}(B,A)$ satisfying $t_{i}[1]\circ e\neq0$. Then
\begin{align}
& l[2]\circ((v_{F}+s_{i})\cup(v_{F}+s_{i}))\nonumber\\=& l[2]\circ v_{F}[1]\circ k[1]\circ t_{i}\circ l+l[2]\circ k[2]\circ t_{i}[1]\circ l[1]\circ v_{F}.\nonumber
\end{align}
Since $l[2]\circ v_{F}[1]\circ k[1]=e[1]$ and $l[2]\circ k[2]=0$, we have $l[2]\circ((v_{F}+s_{i})\cup(v_{F}+s_{i}))=e[1]\circ t_{i}\circ l$. Notice that $e[1]\circ t_{i}\in\mathrm{Ext}^{2}(B,B)=0$, so $l[2]\circ((v_{F}+s_{i})\cup(v_{F}+s_{i}))=0$. On the other hand, $s_{4}$ is a nontrivial element in $N_{H/\mathbb{P}^{3}\times(\mathbb{P}^{3})^{*},E}$. By Lemma \ref{ri}, $s_{4}$ can be completed to the following commutative diagram with $e[1]\circ t_{4}+r_{4}[1]\circ e\neq0$ in $\mathrm{Ext}^{2}(A,B)$:

\begin{center}
$\begin{CD}
A @>k>> F @>l>> B\\
@Vt_{4}VV @Vs_{4}VV @Vr_{4}VV\\
A[1] @>k[1]>> F[1] @>l[1]>> B[1]
\end{CD}$
\end{center}

\noindent Now
\begin{align}
& l[2]\circ((v_{F}+s_{4})\cup(v_{F}+s_{4}))\circ k\nonumber\\=& l[2]\circ v_{F}[1]\circ s_{4}\circ k+ l[2]\circ s_{4}[1]\circ v_{F}\circ k\nonumber\\=& l[2]\circ v_{F}[1]\circ k[1]\circ t_{4}+r_{4}[1]\circ l[1]\circ v_{F}\circ k\nonumber\\=& e[1]\circ t_{4}+r_{4}[1]\circ e\neq0.\nonumber
\end{align}By the diagram in Proposition \ref{cao}, $k^{*}:\mathrm{Ext}^{2}(F,B)\longrightarrow\mathrm{Ext}^{2}(A,B)$ is an isomorphism, hence $l[2]\circ((v_{F}+s_{4})\cup(v_{F}+s_{4}))\neq0$.
\end{proof}

\begin{corol}
The two irreducible components of $\mathbf{M}_{2}$ intersect transeversely.
\end{corol}
\begin{proof}
Proposition \ref{gan} shows that $\kappa_{2}^{-1}(0)$ is cut out by equations $u_{1}u_{2},u_{1}u_{3},u_{1}u_{4},u_{1}u_{5}$ in $\mathrm{Ext}^{1}(F,F)$, so all first order deformations that can be lifted to the second order form a $\mathbb{C}^{15}\cup\mathbb{C}^{12}$ satisfying $\mathbb{C}^{15}\cap\mathbb{C}^{12}=\mathbb{C}^{11}$ in $\mathrm{Ext}^{1}(F,F)$. But $T_{\varphi_{F},F}(T_{\mathbf{P},F})\cup T_{\delta,F}(T_{\mathbf{B},F})=\mathbb{C}^{15}\cup\mathbb{C}^{12}$ and $T_{\varphi_{F},F}(T_{\mathbf{P},F})\cap T_{\delta,F}(T_{\mathbf{B},F})=T_{\varphi_{F},F}(T_{\mathbb{P}(\mathcal{N}_{H/\mathbf{M}_{1}}^{*}),F})=\mathbb{C}^{11}$ by Remark \ref{mafuyu} (2), so indeed we have exhibited all versal deformations of $F$ and the two components of $\mathbf{M}_{2}$ intersect transversely.
\end{proof}

We end this section by proving $\mathbf{M}_{2}$ is a projective variety.
\begin{theorem}
The moduli space $\mathbf{M}_{2}$ is a projective variety.
\end{theorem}
\begin{proof}
$\mathbf{M}_{2}$ is smooth outside the intersection of its two components by Remark \ref{fadian} and Remark \ref{mafuyu} (1) . For any $F\in\mathbb{P}(\mathcal{N}_{H/\mathbf{M}_{1}}^{*})$, since no first order deformation other than a versal one can be lifted to the second order, $\mathbf{M}_{2}$ is reduced at $F$. This proves $\mathbf{M}_{2}$ is reduced. Now we can view $\mathbf{M}_{2}$ as the pushout of the closed embeddings $\mathbf{B}\longleftarrow\mathbb{P}(\mathcal{N}_{H/\mathbf{M}_{1}}^{*})\longrightarrow \mathbf{P}$. In general a pushout diagram does not exist in the category of schemes, but when the two morphisms are closed embeddings it exists [SchK05, Lemma 3.9]. This proves that $\mathbf{M}_{2}$ is a scheme. The fact that $\mathbf{M}_{2}$ is projective and of finite type comes after the analysis of wall-crossing $(3)$ in the next section, where we prove that $\mathbf{M}_{3}$ is a blow-up of $\mathbf{M}_{2}$ along a smooth center contained in $\varphi_{F}(\mathbf{P})\setminus\delta(\mathbf{B})$. Since $\mathbf{M}_{3}$ is the Hilbert scheme, it is automatically projective and of finite type, so $\mathbf{M}_{2}$ is a projective variety.
\end{proof}
\section{The Third Wall-crossing}
In this section, we study the third wall-crossing and prove $(4)$ in the Main theorem. To be more precise, we will prove the following theorem. Let $V$ be a plane in $\mathbb{P}^{3}$ and $q$ be a point on $V$.
\begin{theorem}
The third wall-crossing is simple with a family of pairs of destabilizing objects $(\mathcal{O}(-1)$, $\mathcal{I}_{q/V}(-3))$.
The moduli space of semistable objects after the wall-crossing is the Hilbert scheme of twisted cubics $\mathbf{M}_{3}$. $\mathbf{M}_{3}$ is also the blow-up of $\mathbf{M}_{2}$ along a $5$-dimensional smooth center contained in $\mathbf{P}\setminus\mathbf{B}$.
\label{zhu2}
\end{theorem}
We fix the family of pairs of destabilizing objects to be
\begin{equation}
\left(A,B\right)=\left(\mathcal{O}(-1),\mathcal{I}_{q/V}(-3)\right),\nonumber
\end{equation}
The method is almost the same with the previouse section, but the situation here is easier since we expect no extra components or singularities occur and $\mathbf{M}_{3}$ is a blow-up of $\mathbf{M}_{2}$ along a smooth center.

The following Hom and Ext group computations are straightforward.
\begin{lemma}
$\mathrm{Hom}(A,B)=\mathrm{Hom}(B,A)=0$, $\mathrm{Hom}(A,A)=\mathrm{Hom}(B,B)=\mathbb{C}$;

$\mathrm{Ext}^{1}(A,B)=\mathbb{C}$, $\mathrm{Ext}^{1}(A,A)=0$, $\mathrm{Ext}^{1}(B,B)=\mathbb{C}^{5}$, $\mathrm{Ext}^{1}(B,A)=\mathbb{C}^{10}$;
 
$\mathrm{Ext}^{2}(A,B)=0$, $\mathrm{Ext}^{2}(B,B)=\mathbb{C}^{2}$, $\mathrm{Ext}^{2}(A,A)=0$, $\mathrm{Ext}^{2}(B,A)=\mathbb{C}$;

$\mathrm{Ext}^{3}(A,B)=\mathrm{Ext}^{3}(A,A)=\mathrm{Ext}^{3}(B,B)=\mathrm{Ext}^{3}(B,A)=0$.
\end{lemma}
Similar to Proposition \ref{dadiao}, the incidence hyperplance $H$ is the moduli space of nontrivial extensions of $A$ by $B$. Similar to Proposition \ref{zhongyao}, we can construct an embedding $\varphi'_{E}:H\longrightarrow \mathbf{M}_{2}$. Since $\mathbf{M}_{2}$ has two irreducible components $\mathbf{B}$ and $\mathbf{P}$, we want to know which component $H$ lies in.
\begin{propo}
Under the induced morphism $\varphi_{E}'$, $H$ is embedded into $\mathbf{P}\setminus\mathbf{B}$.
\end{propo}
\begin{proof}
Take any $E\in H$, we have a nontrivial extension $0\longrightarrow B\longrightarrow E\longrightarrow A\longrightarrow0$.
 By using long exact sequences for $\mathrm{Hom}$ functor, we get the following commutative diagram with exact rows and columns, and all boundary homomorphims are $0$.

\begin{center}
$\begin{CD}
\mathrm{Ext}^{1}(A,B)=\mathbb{C} @>0>> \mathrm{Ext}^{1}(E,B)=\mathbb{C}^{5} @>>> \mathrm{Ext}^{1}(B,B)=\mathbb{C}^{5}\\
@V0VV @VVV @VVV\\
\mathrm{Ext}^{1}(A,E)=0 @>>> \mathrm{Ext}^{1}(E,E) @>>> \mathrm{Ext}^{1}(E,B)\\
@VVV @VVV @VVV\\
\mathrm{Ext}^{1}(A,A)=0 @>>> \mathrm{Ext}^{1}(E,A)=\mathbb{C}^{10} @>>> \mathrm{Ext}^{1}(B,A)=\mathbb{C}^{10}\\
@VVV @VVV @VVV\\
\mathrm{Ext}^{2}(A,B)=0 @>>> \mathrm{Ext}^{2}(E,B)=\mathbb{C}^{2} @>>> \mathrm{Ext}^{2}(B,B)=\mathbb{C}^{2}\\
@VVV @VVV @VVV\\
0 @>>> \mathrm{Ext}^{2}(E,E) @>>> \mathrm{Ext}^{2}(B,E)\\
@VVV @VVV @VVV\\
0 @>>> \mathrm{Ext}^{2}(E,A)=\mathbb{C} @>>> \mathrm{Ext}^{2}(B,A)=\mathbb{C}\\
\end{CD}$
\end{center}

\noindent If $E\in \mathbf{B}\setminus\mathbf{P}$, then $\mathrm{Ext}^{1}(E,E)=\mathbb{C}^{12}$, but this violates the exactness of the central column of the above diagram. If $E\in \mathbf{P}\cap\mathbf{B}$, then by Proposition $4.9$ we have $\mathrm{Ext}^{1}(E,E)=\mathbb{C}^{16}$ and $\mathrm{Ext}^{2}(E,E)=\mathbb{C}^{4}$, which also does not fit into the above diagram. Hence $E\in \mathbf{P}\setminus\mathbf{B}$.
\end{proof}
\begin{remark}
This proposition means that the third wall-crossing only modifies one irreducible component of $\mathbf{M}_{2}$, namely $\mathbf{P}$. It does not touch the other component $\mathbf{B}$.
\end{remark}

On the other hand, we can construct a morphism $\varphi_{F}':\mathbf{P}'\longrightarrow \mathbf{M}_{3}$ that is injective on the level of sets and Zariski tangent spaces, where $\mathbf{P}'$ is a $\mathbb{P}^{9}$-bundle over $H$ parametrizing all nontrivial extensions of $B$ by $A$. This implies that for any $F$ in the image of $\varphi_{F}'$, $\mathrm{Ext}^{1}(F,F)$ is at least $14$-dimensional since $\mathrm{dim}\mathbf{P}'=14$ and $\mathbf{P}'$ is smooth.

If we denote the blow-up of $\mathbf{M}_{2}$ along $H$ by $\mathbf{B}'$, then we can perform the elementary modification on the pullback of the universal family over $\mathbf{M}_{2}$ along the exceptional divisor of $\mathbf{B}'$ to get a flat family $\mathcal{K}'$. Similar to Proposition \ref{ruyao}, $\mathcal{K}'$ induces a morphism $\delta':\mathbf{B}'\longrightarrow \mathbf{M}_{3}$ which is injective on the level of sets and Zariski tangent spaces.
\begin{theorem}
The induced morphism $\delta'$ is an isomorphism.
\label{haoxiang}
\end{theorem}
\begin{proof}
$\mathcal{K}'$ is the same as the universal family over $\mathbf{M}_{2}$ outside the exceptional divisor, so $\delta'$ is an isomorphism outside the exceptional divisor. For any $F$ lying in the exceptional divisor, $\delta'$ induces an injection $T_{\mathbf{B}',F}\longrightarrow\mathrm{Ext}^{1}(F,F)=T_{\mathbf{M}_{3},F}$. To prove $\delta'$ is an isomorphism at $F$, we only need to show $\mathrm{Ext}^{1}(F,F)=\mathbb{C}^{15}=T_{\mathbf{B}',F}$. Since we have an exact sequence $0\longrightarrow A\longrightarrow F\longrightarrow B\longrightarrow0$, this can be done by writing down the long exact sequences for $\mathrm{Hom}$ functor again.

\begin{center}
$\begin{CD}
\mathrm{Ext}^{1}(B,A)=\mathbb{C}^{10} @>>> \mathrm{Ext}^{1}(F,A)=\mathbb{C}^{9} @>>> \mathrm{Ext}^{1}(A,A)=0\\
@VVV @VVV @VVV\\
\mathrm{Ext}^{1}(B,F)=\mathbb{C}^{14} @>>> \mathrm{Ext}^{1}(F,F)=\mathbb{C}^{15} @>>> \mathrm{Ext}^{1}(A,F)=\mathbb{C}\\
@VVV @VVV @VVV\\
\mathrm{Ext}^{1}(B,B)=\mathbb{C}^{5} @>>> \mathrm{Ext}^{1}(F,B)=\mathbb{C}^{6} @>>> \mathrm{Ext}^{1}(A,B)=\mathbb{C}\\
@V0VV @V0VV @VVV\\
\mathrm{Ext}^{2}(B,A)=\mathbb{C} @>>> \mathrm{Ext}^{2}(F,A)=\mathbb{C} @>>> \mathrm{Ext}^{2}(A,A)=0\\
@VVV @VVV @VVV\\
\mathrm{Ext}^{2}(B,F)=\mathbb{C}^{3} @>>> \mathrm{Ext}^{2}(F,F)=\mathbb{C}^{3} @>>> \mathrm{Ext}^{2}(A,F)=0\\
@VVV @VVV @VVV\\
\mathrm{Ext}^{}(B,B)=\mathbb{C}^{2} @>>> \mathrm{Ext}^{2}(F,B)=\mathbb{C}^{2} @>>> \mathrm{Ext}^{2}(A,B)=0
\end{CD}$
\end{center}
\end{proof}

\end{document}